\documentclass[12pt,reqno,oneside]{amsart}
\usepackage{amsfonts}
\usepackage{amsmath, amssymb}
\usepackage{hyperref}

\numberwithin{equation}{section} 
\newtheorem{theorem}{Theorem}[section]
\newtheorem{corollary}{Corollary}[section]
\newtheorem{lemma}{Lemma}[section]
\theoremstyle{definition}
\newtheorem{definition}{Definition}[section]
\theoremstyle{remark}
\newtheorem{remark}{Remark}[section]
\numberwithin{equation}{section}

\thanks{$\ast$ Corresponding Author: N. Magesh ($nmagi\_2000@yahoo.co.in$)\\ The second author is thankful to Prof. S. R. Swamy, {Department of Computer Science and Engineering, R. V. College of Engineering, Bangalore-560 059, Karnataka, India, for his constant support and guidance.}}
\begin{document}
\title[bi-univalent functions with $\kappa-$Fibonacci numbers]{Initial estimates for certain subclasses  of bi-univalent functions with $\kappa-$Fibonacci numbers}
\author{N. Magesh$^{1, \ast}$, J. Nirmala$^{2}$ and J. Yamini$^{3}$}
\maketitle
\begin{center}
	{$^{\ast,\; 1}$Post-Graduate and Research Department of Mathematics,\\
	Government Arts College for Men,\\
	Krishnagiri 635001, Tamilnadu, India.\\
	\texttt{e-mail:} $nmagi\_2000@yahoo.co.in$}\\
	\texttt{ORCID Address:} http://orcid.org/0000-0002-0764-8390.\\
	$^{2}${Department of Mathematics,\; 
	Maharani's Science College for Women\\
	Bangalore-560 0001, Karnataka, India.\\
	{\bf e-mail~:~} $nirmalajodalli@gmail.com.$}\\
\texttt{ORCID Address:} http://orcid.org/0000-0002-1048-5609.\\
$^{3}${Department of Mathematics, Govt First Grade College\\ Vijayanagar, Bangalore-560104,\; Karnataka, India.\\
	{\bf e-mail~:~} $yaminibalaji@gmail.com$.}\\
	\texttt{ORCID Address:} http://orcid.org/0000-0002-3544-7588.
\end{center}
\date{}
\begin{abstract}
	In this  work, we consider certain class of bi-univalent functions related with shell-like curves related to $\kappa-$Fibonacci numbers. Further, we obtain the estimates of initial Taylor-Maclaurin coefficients (second and third coefficients) and Fekete - Szeg\"{o}  inequalities. Also we discuss the special cases of the obtained results. 
\end{abstract}
\maketitle

%%%=====================

\section{Introduction and definitions}
Let $\mathbb{R}=\left( -\infty ,\infty \right) $\ be the set of real
numbers,  and%
\begin{equation*}
\mathbb{N}:=\left\{ 1,2,3,\ldots \right\} =\mathbb{N}_{0}\backslash \left\{
0\right\}
\end{equation*}%
be the set of positive integers. Let $\mathcal{A}$ denote the class of functions of the form
\begin{equation}
f(z)=z+\sum\limits_{n=2}^{\infty }a_{n}z^{n}  \label{Int-e1}
\end{equation}%
which are analytic in the open unit disk $\mathbb{D}=\{z:z\in \mathbb{C}\,\,%
\mathrm{and}\,\,|z|<1\},$ where $\mathbb{C}$ be the set of complex numbers. Further, by $\mathcal{S}$ we shall denote the
class of all functions in $\mathcal{A}$ which are univalent in $\mathbb{D}.$

\par Let $\mathcal{P}$ denote the class of functions of the form
\begin{align*}
p(z)=1+p_1z+p_2z^2+p_3z^3+\ldots, \qquad z \in \mathbb{D}
\end{align*}
which are analytic with $\Re~\{p(z)\}>0.$ Here $p$ is called as Caratheodory functions \cite{PLD}. It is well known that the following correspondence between the class $\mathcal{P}$ and the class of Schwarz functions $w$ exists: $p \in \mathcal{P}$ if and only if $p(z)={1+w(z)}\left/\right.{1-w(z)}.$ Let $\mathcal{P}(\beta),$ $0 \leq \beta < 1,$ denote the class of analytic functions $p$ in $\mathbb{D}$ with $p(0)=1$ and $\Re\left\{p(z)\right\} > \beta.$ 

\par For analytic functions $f$ and $g$ in $\mathbb{D},$ $f$ is said to be
subordinate to $g$ if there exists an analytic function $w$ such that %(see,
%for example, \cite{Miller})%
\begin{equation*}
w(0)=0, \quad \quad |w(z)|<1 \quad \mathrm{and} \quad f(z)=g(w(z)), \qquad
z\in\mathbb{D},
\end{equation*}
denoted by
\begin{equation*}
f \prec g, \qquad
z\in\mathbb{D}
\end{equation*}
or, conventionally, by
\begin{equation*}
f(z) \prec g(z), \qquad
z\in\mathbb{D}.
\end{equation*}

In particular, when $g$ is univalent in $\mathbb{D},$
\begin{equation*}
f \prec g \qquad (z\in\mathbb{D}) ~\Leftrightarrow ~f(0)=g(0) \quad \mathrm{%
	and} \quad f(\mathbb{D}) \subset g(\mathbb{D}).
\end{equation*}
%%=====================================
Some of the important and well-investigated subclasses of $\mathcal{S}$ include (for example) the class $\mathcal{S}%
^{\ast }(\alpha )$ of starlike functions of order $\alpha $ $(0\leqq \alpha
<1)$ in $\mathbb{D}$ and the class $\mathcal{K}(\alpha )$ of convex
functions of order $\alpha $ $(0\leqq \alpha <1)$ in $\mathbb{D},$ 
the class $\mathcal{S}^{\ast }(\varphi )$ of Ma-Minda starlike functions and
the class $\mathcal{K}(\varphi )$ of Ma-Minda convex functions, where $\varphi $
is an analytic function with positive real part in $\mathbb{D}$, $\varphi
(0)=1,$ $\varphi ^{\prime }(0)>0$ and $\varphi $ maps $\mathbb{D}$ onto a
region starlike with respect to 1 and symmetric with respect to the real
axis) (see \cite{PLD}. %The above-defined function classes
%have recently been investigated rather extensively in (for example) \cite%
%{YCK-HMS,ravi-Pola-2005,HMS-QHX-GPW,Xu-Cai-HMS} and the references therein.
%%=====================================

The inverse functions of the functions in the class $\mathcal{S}$ may not be defined on the entire unit disc $\mathbb{D}$ although the functions in the class $\mathcal{S}$ are invertible. However using Koebe one quarter theorem \cite{PLD} it is obvious that the image of $\mathbb{D}$ under every function $f \in \mathcal{S}$ contains a disc of radius $1/4.$ Hence every function $f\in \mathcal{S}$ has an inverse $%
f^{-1},$ defined by
\begin{equation*}
f^{-1}(f(z))=z, \qquad
z\in\mathbb{D}
\end{equation*}
and
\begin{equation*}
f(f^{-1}(w))=w, \qquad |w| < r_0(f);\,\, r_0(f) \geqq \frac{1}{4},
\end{equation*}
where
\begin{equation*}
f^{-1}(w) = w - a_2w^2 + (2a_2^2-a_3)w^3 - (5a_2^3-5a_2a_3+a_4)w^4+\ldots .
\end{equation*}

A function $f\in \mathcal{A}$ is said to be bi-univalent in $\mathbb{D}$ if
both $f$ and $f^{-1}$ are univalent in $\mathbb{D}.$ Let $\Sigma $
denote the class of bi-univalent functions in $\mathbb{D}$ given by (\ref%
{Int-e1}). 
%%=================================
Various subclasses of $\Sigma $ were introduced and non-sharp
estimates on the first two coefficients $|a_{2}|$ and $|a_{3}|$ in the
Taylor-Maclaurin series expansion (\ref{Int-e1}) were found in several
recent investigations (see, for example, \cite%
{CA-NM-JY-NBG-2020-JA,Ali-Ravi-Ma-Mina-class,SA-SY-BMM-2018,Bulut,Caglar-Orhan,Deniz,VBG-SBJ-Ganita-2018,Jay-SGH-SAH-2014,HONMVKBAfr-2015,HO-NM-VKB-TMJ-2017,HO-NM-VKB-2019-AEJM,HMS-DB-2015,HMS-Caglar, HMS-SSE-RMA-2015, HMS-SG-FG-2015, HMS-AKM-PG,HMS-GMS-NM-2013,Tang,Tang-NM-VKB-CA-2019-JIA,Zaprawa}
and references therein). However, the problem is to find the coefficient bounds on $|a_{n}|
$ ($n=3,4,\dots $) for functions $f\in \Sigma $ is still an open problem.
%%-----------------------------

\par Falcon and Plaza \cite{SF-AP-2006-CSF} have introduced the Fibonacci $\kappa-$ numbers as below. The recurrence relation for the $\kappa-$Fibonacci  sequence $\left\{F_{\kappa,\; n}\right\}_{n=0}^{\infty}$, where $\kappa \in \mathbb{R}$ as defined as 
\[
F_{\kappa,\; n+1} = \kappa F_{\kappa,\; n} + F_{\kappa,\; n-1}
\]
for $n \geq 1$ with the initial conditions $F_{\kappa,\; 0}=0,$ $F_{\kappa,\; 1}=1$ and 
$F_{\kappa,\; n}=\dfrac{(\kappa - \tau_{\kappa})^{n} - \tau_{\kappa}^{n}}{\sqrt{\kappa^{2}}+4}$ with $\tau_{\kappa} = \dfrac{\kappa - \sqrt{\kappa^{2}+4}}{2}, \; z \in \mathbb{D},$ where $F_{\kappa,\; n}$ represents the $n^{\hbox{th}}$ element of the $\kappa-$Fibonacci sequence. 

\par Let $\{u_{n}\}$ be the sequence of Fibonacci numbers. The recurrence relation for Fibonacci sequence is given by 
\[
u_{n+2} = u_{n+1} + u_{n}, \qquad n = 0,\; 1,\; 2,\; \cdots
\]
with the initial conditions $u_{0}=0$ and $u_{1}=1.$

\par Sokol et al. \cite{JS-RKR-NYO-2015-HJMS} proved that if $\tilde{p}_{\kappa}(z) = \dfrac{1+\tau_{\kappa}^{2}z^{2}}{1-\kappa \tau_{\kappa}z - \tau_{\kappa}^{2}z^{2}} = 1 + \sum\limits_{n=1}^{\infty} \tilde{p}_{\kappa,\; n}z^{n},$ then we have $\tilde{p}_{\kappa,\; n} = (F_{\kappa,\; n-1} + F_{\kappa,\; n+1})\tau_{\kappa}^{n},$ $n \geq 1.$ This shows the relevant connection between the Fibonacci $\kappa-$numbers and the coefficients of $\tilde{p}_{\kappa}(z).$ That is,
\begin{eqnarray*}
\tilde{p}_{\kappa}(z) 
&=& \dfrac{1+\tau_{\kappa}^{2}z^{2}}{1-\kappa \tau_{\kappa}z - \tau_{\kappa}^{2}z^{2}} \\
&=& 1 + \sum\limits_{n=1}^{\infty} \tilde{p}_{\kappa,\; n}z^{n}\\
&=& 1 + (F_{\kappa,\; 0} + F_{\kappa,\; 2}) \tau_{\kappa} z 
	  + (F_{\kappa,\; 1} + F_{\kappa,\; 3}) \tau_{\kappa}^{2} z^{2} 
	  + \cdots \\
&=& 1 + \kappa \tau_{\kappa} z + (\kappa^{2} + 2) \tau_{\kappa}^{2} z^{2}
	  + (\kappa^{3} + 3\kappa) \tau_{\kappa}^{3} z^{3}
	  + \cdots,
\end{eqnarray*}
where $\tau_{\kappa} = \dfrac{\kappa - \sqrt{\kappa^{2}+4}}{2}, \; z \in \mathbb{D}$ (see \cite{NYO-JS-2015-BMMSS}).
%%-----------------------------
\par The classes $\mathcal{S}\mathcal{L}(\tilde{p})$ and $\mathcal{K}\mathcal{S}\mathcal{L}(\tilde{p})$ of shell-like functions and convex shell-like functions are respectively, characterized by $zf^{\prime}\left/\right.f(z)\prec \tilde{p}(z)$ or $1+z^{2}f^{\prime\prime}\left/\right.f^{\prime}(z)\prec \tilde{p}(z),$ where $\tilde{p}(z)={(1+\tau^{2}z^{2})}\left/\right.{(1-\tau z - \tau^{2}z^{2})},$ $\tau = {(1-\sqrt{5})}\left/\right.{2} \approx -0.618.$ The classes $\mathcal{S}\mathcal{L}(\tilde{p})$ and $\mathcal{K}\mathcal{S}\mathcal{L}(\tilde{p})$ were introduced and studied by Sok\'{o}{\l} \cite{JS-1999} and Dziok et al. \cite{JD-RKR-JS-CMA-2011} respectively (see also \cite{JD-RKR-JS-AMC-2011,RKR-JS-2016}). 

%The function $\tilde{p}$ is not univalent in $\mathbb{D},$ but it is univalent in the disc $|z| < {(3-\sqrt{5})}\left/\right.{2} \approx 0.38.$ For example, $\tilde{p}(0) = \tilde{p}\left({-1}\left/\right.{2\tau}\right)=1$ and $\tilde{p}\left(e^{\mp}\arccos\left(1/4\right)\right)={\sqrt{5}}\left/\right.{5}$ and it may also be noticed that ${1}\left/\right.{\left|\tau\right|} = {\left|\tau\right|}\left/\right.{1-\left|\tau\right|}$ which shows that the number $\left|\tau\right|$ divides $\left[0,\; 1\right]$ such that it fulfills the golden section. The image of the unit circle $\left|z\right|=1$ under $\tilde{p}$ is a curve described by the equation given by $\left(10x-\sqrt{5}\right)y^{2} = \left(\sqrt{5}-2x\right)\left(\sqrt{5}x-1\right)^{2},$ which is translated and revolved trisectrix of Maclaurin. The curve $\tilde{p}\left(re^{it}\right)$ is a closed curve without any loops for $0<r\leq r_{0}={(3-\sqrt{5})}\left/\right.{2} \approx 0.38.$ For $r_{0}<r<1,$ it has a loop and for $r=1$, it has a vertical asymptote. Since $\tau$ satisfies the equation $\tau^{2}=1+\tau,$  this expression can be used to obtain higher powers $\tau^{n}$ as a linear function of lower powers, which in turn can be decomposed all the way down to a linear combination of $\tau$ and $1.$ The resulting recurrence relationships yield Fibonacci numbers $u_{n}$
%\[
%\tau^{n} = u_{n}\tau+u_{n-1}.
%\]
Recently Raina and Sok\'{o}{\l} \cite{RKR-JS-2016}, taking $\tau z =t,$ showed that
\begin{eqnarray}
\tilde{p}(z) 
&=& \dfrac{1+\tau^{2}z^{2}}{1-\tau z - \tau^{2}z^{2}} %\nonumber\\ %= %\left(t+\dfrac{1}{t}\right)\dfrac{t}{1-t-t^{2}}\nonumber\\
%&=& \dfrac{1}{\sqrt{5}}\left(t+\dfrac{1}{t}\right)\left(\dfrac{1}{1-(1-\tau)t} - %\dfrac{1}{1-\tau t}\right)\nonumber\\
%&=& %\left(t+\dfrac{1}{t}\right)\sum\limits_{n=1}^{\infty}\dfrac{(1-\tau)^{n}-\tau^{n}}{\s%qrt{5}}t^{n}\nonumber\\
%&=& \left(t+\dfrac{1}{t}\right)\sum\limits_{n=1}^{\infty}u_{n}t^{n}\nonumber\\
= 1 + \sum\limits_{n=1}^{\infty} \left(u_{n-1}+u_{n+1}\right)\tau^{n}z^{n},
\end{eqnarray}
where
\begin{eqnarray}
u_{n}=\dfrac{(1-\tau)^{n}-\tau^{n}}{\sqrt{5}},\qquad \tau = \dfrac{1-\sqrt{5}}{2}, \qquad n=1,\; 2,\; \ldots.
\end{eqnarray}
This shows that the relevant connection of $\tilde{p}$ with the sequence of Fibonacci numbers $u_{n}$, such that 
\[
u_{0} = 0,\quad u_{1}=1, \quad u_{n+2} = u_{n}+u_{n+1}
\]
for $n=0,\; 1,\; 2,\; 3,\; \ldots~.$ Hence
\begin{eqnarray}
\tilde{p}(z) %&=& 1 + \sum\limits_{n=1}^{\infty} \tilde{p}_{n}z^{n} \nonumber\\
%&=& 1 + \left(u_{0}+u_{2}\right)\tau z + \left(u_{1}+u_{3}\right)\tau^{2} z^{2}
%	+ \sum\limits_{n=3}^{\infty}\left(u_{n-3}+u_{n-2}+u_{n-1}+u_{n}\right)\tau^{n}z^{n}\nonumber\\
&=& 1 + \tau z + 3\tau^{2}z^{2} + 4\tau^{3}z^{3}+7\tau^{4}z^{4}+11\tau^{5}z^{5}+\ldots.   
\end{eqnarray}
We note that the function $\tilde{p}$ belongs to the class $P(\beta)$ with $\beta = \dfrac{\sqrt{5}}{10} \approx 0.2236$ (see \cite{RKR-JS-2016}).
%\begin{theorem}\cite{RKR-JS-2016}
%	The function 
%	\[
%	\tilde{p}(z) = \dfrac{1+\tau^{2}z^{2}}{1-\tau z - \tau^{2}z^{2}}
%	\]
%	belongs to the class $P(\beta)$ with $\beta = \dfrac{\sqrt{5}}{10} \approx 0.2236.$
%\end{theorem}
%%--------------------------------------------

The classes $\mathcal{S}\mathcal{L}^{\kappa}(\tilde{p}_{\kappa})$ and $\mathcal{K}\mathcal{S}\mathcal{L}^{\kappa}(\tilde{p}_{\kappa})$  characterized by $zf^{\prime}\left/\right.f(z)\prec \tilde{p}_{\kappa}(z)$ or $1+z^{2}f^{\prime\prime}\left/\right.f^{\prime}(z)\prec \tilde{p}_{\kappa}(z),$ where $\tilde{p}_{\kappa}(z)={(1+\tau_{\kappa}^{2}z^{2})}\left/\right.{(1-\kappa\tau_{\kappa} z - \tau_{\kappa}^{2}z^{2})},$ $\tau_{\kappa} = {(\kappa-\sqrt{\kappa^{2}+4})}\left/\right.{2}.$ The classes $\mathcal{S}\mathcal{L}^{\kappa}(\tilde{p}_{\kappa})$ and $\mathcal{K}\mathcal{S}\mathcal{L}^{\kappa}(\tilde{p}_{\kappa})$ were introduced and studied by Yilmaz \"{O}zg\"{u}r and Sok\'{o}{\l} \cite{NYO-JS-2015-BMMSS} and G\"{u}ney et al. \cite{HOG-JS-IS-2018-AUA} respectively (see also \cite{JS-RKR-NYO-2015-HJMS}). It was proved in \cite{NYO-JS-2015-BMMSS} that functions in the class $\mathcal{S}\mathcal{L}^{\kappa}(\tilde{p}_{\kappa})$ are univalent in $\mathbb{D}.$ Moreover,
the class $\mathcal{S}\mathcal{L}^{\kappa}(\tilde{p}_{\kappa})$ is a subclass of the class of starlike functions $\mathcal{S}^{\ast}$ even more, starlike of order ${(\kappa{(\kappa^{2}+4)}^{-1/2})}\left/\right.{2}.$

For $\kappa = 1$, the classes  $\mathcal{S}\mathcal{L}(\tilde{p})$ and $\mathcal{K}\mathcal{S}\mathcal{L}(\tilde{p})$ were introduced and studied by Sok\'{o}{\l} \cite{JS-1999} and Dziok et al. \cite{JD-RKR-JS-CMA-2011} respectively (see also \cite{JD-RKR-JS-AMC-2011,VSM-AE-SY-2019-MS,RKR-JS-2016}) and the references therein. Also, one can refer \cite{SA-SY-SC-2019-Math, SA-GH-JMJ-2019-arXiv, HOG-2019-SCMA,NM-VKB-CA-2016} and references therein for other subclasses of bi-univalent functions which are related with shell-like curves connected with Fibonacci numbers.

%%--------------------------------------------

Recently, in literature, the initial coefficient estimates are found for functions in the class of bi-univalent functions associated with certain polynomials like the Faber polynomial, the Lucas polynomial, the Chebyshev polynomial and the Gegenbauer
polynomial. Motivated in this line, inspired by the works of Ali et al. \cite{Ali-Ravi-Ma-Mina-class}, G\"{u}ney \cite{HOG-2019-SCMA}, G\"{u}ney et al. \cite{HOG-GMS-JS-2019-CFS-kF} and Magesh et al. \cite{NM-VKB-CA-2016}, we introduce the following new subclasses of bi-univalent functions, as follows:

\begin{definition}
\label{def1.1} A function $f\in \Sigma$ of the form 
\begin{equation*}
f(z) = z + \sum\limits_{n=2}^{\infty} a_{n} z^{n}, 
\end{equation*}
belongs to the class $\mathcal{W}\mathcal{S}\mathcal{L}_{\Sigma}^{\kappa}(\gamma, \lambda, \alpha, \tilde{p}_{\kappa}),
$ $\gamma\in \mathbb{C}\backslash\{0\},$ $\alpha \geq 0$ and $\lambda \geq 0,$  if the following conditions are satisfied: 
\begin{equation}  \label{CR-RAP-NM-P1-e1}
1+\frac{1}{\gamma}\left((1-\alpha+2\lambda)\frac{f(z)}{z}+(\alpha-2%
\lambda)f^{\prime }(z)+\lambda zf^{\prime \prime }(z)-1\right) \prec \widetilde{p_{\kappa}(z)} = \dfrac{1+\tau^{2}_{\kappa}z^{2}}{1-\kappa\tau_{\kappa} z - \tau^{2}_{\kappa}z^{2}},\; z\in\mathbb{D},
\end{equation}
and for $g(w)=f^{-1}(w)$ 
\begin{equation}  \label{CR-RAP-NM-P1-e2}
1+\frac{1}{\gamma}\left((1-\alpha+2\lambda)\frac{g(w)}{w}+(\alpha-2%
\lambda)g^{\prime }(w)+\lambda wg^{\prime \prime }(w)-1\right) \prec \widetilde{p_{\kappa}(w)}=\dfrac{1+\tau^{2}_{\kappa}w^{2}}{1-\kappa\tau_{\kappa} w - \tau^{2}_{\kappa}w^{2}},\; w\in\mathbb{D},
\end{equation}
where
$\tau_{\kappa} = \dfrac{\kappa-\sqrt{\kappa^{2}+4}}{2}.$
\end{definition}

%%=========================================================================
It is interesting to note that the special values of $\alpha,$ $\gamma$ and $\lambda$ lead the class $\mathcal{W}\mathcal{S}\mathcal{L}_{\Sigma}^{\kappa}(\gamma,
\lambda, \alpha, \tilde{p}_{\kappa})$ to various subclasses, as following illustrations:

\begin{enumerate}
\item For $\alpha=1+2\lambda,$ we get the class $\mathcal{W}\mathcal{S}\mathcal{L}
_{\Sigma}^{\kappa}(\gamma, \lambda, 1+2\lambda, \tilde{p}_{\kappa})\equiv \mathcal{F}\mathcal{S}\mathcal{L}_{\Sigma}^{\kappa}(\gamma, \lambda, \tilde{p}_{\kappa}).$ A function $f\in \Sigma$ of the form 
\begin{equation*}
f(z) = z + \sum\limits_{n=2}^{\infty} a_{n} z^{n}, 
\end{equation*}
is said to be in $\mathcal{F}\mathcal{S}\mathcal{L}_{\Sigma}^{\kappa}(\gamma, \lambda, \tilde{p}_{\kappa}),$ if the
following conditions 
\begin{equation*}
1+\frac{1}{\gamma}\left(f^{\prime }(z)+\lambda zf^{\prime \prime
}(z)-1\right) \prec \widetilde{p_{\kappa}(z)}=\dfrac{1+\tau^{2}_{\kappa}z^{2}}{1-\kappa\tau_{\kappa} z - \tau^{2}_{\kappa}z^{2}},\;z\in\mathbb{D} 
\end{equation*}
and for $g(w)=f^{-1}(w)$ 
\begin{equation*}
1+\frac{1}{\gamma}\left(g^{\prime }(w)+\lambda wg^{\prime \prime
}(w)-1\right) \prec \widetilde{p_{\kappa}(w)}
=\dfrac{1+\tau^{2}_{\kappa}w^{2}}{1-\kappa\tau_{\kappa} w - \tau^{2}_{\kappa}w^{2}},\;w\in\mathbb{D}, 
\end{equation*}
hold, where
$\tau_{\kappa} = \dfrac{\kappa-\sqrt{\kappa^{2}+4}}{2}.$

\item For $\lambda=0,$ we obtain the class $\mathcal{W}\mathcal{S}\mathcal{L}_{\Sigma}^{\kappa}(\gamma, 0,
\alpha, \tilde{p}_{\kappa})\equiv \mathcal{B}\mathcal{S}\mathcal{L}_{\Sigma}^{\kappa}(\gamma, \alpha, \tilde{p}_{\kappa}).$ A
function $f\in \Sigma$ of the form 
\begin{equation*}
f(z) = z + \sum\limits_{n=2}^{\infty} a_{n} z^{n}, 
\end{equation*}
is said to be in $\mathcal{B}\mathcal{S}\mathcal{L}_{\Sigma}^{\kappa}(\gamma, \alpha, \tilde{p}_{\kappa}),$ if the
following conditions 
\begin{equation*}
1+\frac{1}{\gamma}\left((1-\alpha)\frac{f(z)}{z}+\alpha f^{\prime
}(z)-1\right) \prec \widetilde{p_{\kappa}(z)}=\dfrac{1+\tau^{2}_{\kappa}z^{2}}{1-\kappa\tau_{\kappa} z - \tau^{2}_{\kappa}z^{2}},\;z\in\mathbb{D} 
\end{equation*}
and for $g(w)=f^{-1}(w)$ 
\begin{equation*}
1+\frac{1}{\gamma}\left((1-\alpha)\frac{g(w)}{w}+\alpha g^{\prime
}(w)-1\right)
\prec \widetilde{p_{\kappa}(w)}
=\dfrac{1+\tau^{2}_{\kappa}w^{2}}{1-\kappa\tau_{\kappa} w - \tau^{2}_{\kappa}w^{2}},\;w\in\mathbb{D}
\end{equation*}
hold, where
$\tau_{\kappa} = \dfrac{\kappa-\sqrt{\kappa^{2}+4}}{2}.$

\item For $\lambda=0$ and $\alpha=1,$ we have the class $\mathcal{W}\mathcal{S}\mathcal{L}_{\Sigma}^{\kappa}(\gamma, 0, 1, \tilde{p}_{\kappa})\equiv \mathcal{H}\mathcal{S}\mathcal{L}_{\Sigma}^{\kappa}(\gamma, \tilde{p}_{\kappa}).$
A function $f\in \Sigma$ of the form 
\begin{equation*}
f(z) = z + \sum\limits_{n=2}^{\infty} a_{n} z^{n}, 
\end{equation*}
is said to be in $\mathcal{H}\mathcal{S}\mathcal{L}_{\Sigma}^{\kappa}(\gamma, \tilde{p}_{\kappa}),$ if the following
conditions 
\begin{equation*}
1+\frac{1}{\gamma}\left(f^{\prime }(z)-1\right) \prec \widetilde{p_{\kappa}(z)}=\dfrac{1+\tau^{2}_{\kappa}z^{2}}{1-\kappa\tau_{\kappa} z - \tau^{2}_{\kappa}z^{2}},\;z\in\mathbb{D} 
\end{equation*}
and for $g(w)=f^{-1}(w)$ 
\begin{equation*}
1+\frac{1}{\gamma}\left(g^{\prime }(w)-1\right) \prec 
\widetilde{p_{\kappa}(w)}
=\dfrac{1+\tau^{2}_{\kappa}w^{2}}{1-\kappa\tau_{\kappa} w - \tau^{2}_{\kappa}w^{2}},\;w\in\mathbb{D}
\end{equation*}
hold, where
$\tau_{\kappa} = \dfrac{\kappa-\sqrt{\kappa^{2}+4}}{2}.$
\end{enumerate}
%------------------------bi-Bazelevic---------------------------------
\begin{definition}
	\label{def1.2} A function $f\in \Sigma$ of the form \eqref{Int-e1}
	%\begin{equation*}
	%f(z) = z + \sum\limits_{n=2}^{\infty} a_{n} z^{n}, 
	%\end{equation*}
	belongs to the class $\mathcal{R}\mathcal{S}\mathcal{L}_{\Sigma}^{\kappa}(\gamma, \lambda, \tilde{p}_{\kappa}),$ $\gamma\in \mathbb{C}\backslash\{0\}$ and $\lambda \geq 0,$  if the following conditions are satisfied: 
	\begin{equation}  \label{C1-e1}
	1+\frac{1}{\gamma}\left(\frac{z^{1-\lambda}f^{\prime}(z)}{(f(z))^{1-\lambda}}%
	-1\right) \prec \widetilde{p_{\kappa}(z)}=\dfrac{1+\tau^{2}_{\kappa}z^{2}}{1-\kappa\tau_{\kappa} z - \tau^{2}_{\kappa}z^{2}},\;z\in\mathbb{D} 
	\end{equation}
	and for $g(w)=f^{-1}(w)$ 
	\begin{equation}  \label{C1-e2}
	1+\frac{1}{\gamma}\left(\frac{w^{1-\lambda}g^{\prime}(w)}{(g(w))^{1-\lambda}}%
	-1\right) 
	\widetilde{p_{\kappa}(w)}
	=\dfrac{1+\tau^{2}_{\kappa}w^{2}}{1-\kappa\tau_{\kappa} w - \tau^{2}_{\kappa}w^{2}},\;w\in\mathbb{D}
	\end{equation}
	where
	$\tau_{\kappa} = \dfrac{\kappa-\sqrt{\kappa^{2}+4}}{2}.$
\end{definition}

\begin{enumerate}
	\item For $\lambda=0,$ we let the class $\mathcal{R}\mathcal{S}\mathcal{L}_{\Sigma}^{\kappa}(\gamma,
	0, \tilde{p}_{\kappa})\equiv \mathcal{S}\mathcal{L}_{\Sigma}^{\kappa}(\gamma, \tilde{p}_{\kappa}).$ A function $f\in
	\Sigma$ of the form \eqref{Int-e1}
	%\begin{equation*}
	%f(z) = z + \sum\limits_{n=2}^{\infty} a_{n} z^{n}, 
	%\end{equation*}
	is said to be in $\mathcal{S}\mathcal{L}_{\Sigma}^{\kappa}(\gamma,\; \tilde{p}_{\kappa}),$ if the following
	conditions 
	\begin{equation*}
	1+\frac{1}{\gamma}\left(\frac{zf^{\prime }(z)}{f(z)}-1\right) 
	\prec 
	\widetilde{p_{\kappa}(z)}=\dfrac{1+\tau^{2}_{\kappa}z^{2}}{1-\kappa\tau_{\kappa} z - \tau^{2}_{\kappa}z^{2}},\;z\in\mathbb{D} 
	\end{equation*}
	and for $g(w)=f^{-1}(w)$ 
	\begin{equation*}
	1+\frac{1}{\gamma}\left(\frac{wg^{\prime }(w)}{g(w)}-1\right) 
	\prec 
	\widetilde{p_{\kappa}(w)}
	=\dfrac{1+\tau^{2}_{\kappa}w^{2}}{1-\kappa\tau_{\kappa} w - \tau^{2}_{\kappa}w^{2}},\;w\in\mathbb{D}
	\end{equation*}
	hold, where
	$\tau_{\kappa} = \dfrac{\kappa-\sqrt{\kappa^{2}+4}}{2}.$
	
	\begin{remark}
		For $\gamma=1$  the class $\mathcal{S}\mathcal{L}_{\Sigma}^{\kappa}(1, \tilde{p}_{\kappa})\equiv \mathcal{S}\mathcal{L}_{\Sigma}^{\kappa}(\tilde{p}_{\kappa})$ was introduced and studied  G\"{u}ney et al. \cite{HOG-GMS-JS-2019-CFS-kF} and for $\kappa=1$ the class was studied by G\"{u}ney et al.  \cite{HOG-GMS-JS-Fib-2018}.
	\end{remark}
	
	\item For $\lambda=1,$ we have $\mathcal{R}\mathcal{S}\mathcal{L}_{\Sigma}^{\kappa}(\gamma,
	1, \tilde{p}_{\kappa})\equiv \mathcal{H}\mathcal{S}\mathcal{L}_{\Sigma}^{\kappa}(\gamma, \tilde{p}_{\kappa}).$ 
%	A function $f\in\Sigma$ of the form \eqref{Int-e1}
%	%\begin{equation*}
%	%f(z) = z + \sum\limits_{n=2}^{\infty} a_{n} z^{n}, 
%	%\end{equation*}
%	is said to be in $\mathcal{H}\mathcal{S}\mathcal{L}_{\Sigma}(\gamma,\; \tilde{p}),$ %if the following
%	conditions 
%	\begin{equation*}
%	1+\frac{1}{\gamma}\left(f^{\prime }(z)-1\right) \prec \widetilde{p(z)}=\dfrac{1+\tau^{2}z^{2}}{1-\tau z - \tau^{2}z^{2}},\;z\in\mathbb{D} 
%	\end{equation*}
%	and for $g(w)=f^{-1}(w)$ 
%	\begin{equation*}
%	1+\frac{1}{\gamma}\left(g^{\prime }(w)-1\right) \prec \widetilde{p(w)}=\dfrac{1+\tau^{2}w^{2}}{1-\tau w - \tau^{2}w^{2}},\;w\in\mathbb{D} 
%	\end{equation*}
%	hold, where
%	$\tau_{\kappa} = \dfrac{\kappa-\sqrt{\kappa^{2}+4}}{2}.$
\end{enumerate}
%%----------------------\lambda-pseudo-starlike------------------------------------
\begin{definition}
	\label{def1.3} A function $f\in \Sigma$ of the form 
	\begin{equation*}
	f(z) = z + \sum\limits_{n=2}^{\infty} a_{n} z^{n}, 
	\end{equation*}
	belongs to the class $\mathcal{S}\mathcal{L}\mathcal{B}_{\Sigma}^{\kappa}(\lambda; \tilde{p}_{\kappa}),$  $\lambda \geq 1,$ if the following conditions are satisfied: 
	\begin{equation}  \label{C1-e1}
	\frac{z\left[f^{\prime}(z)\right]^{\lambda}}{f(z)} 
	\prec 
	\widetilde{p_{\kappa}(z)}=\dfrac{1+\tau^{2}_{\kappa}z^{2}}{1-\kappa\tau_{\kappa} z - \tau^{2}_{\kappa}z^{2}},\;z\in\mathbb{D} 
	\end{equation}
	and for $g(w)=f^{-1}(w)$ 
	\begin{equation}  \label{C1-e2}
	\frac{w\left[g^{\prime}(w)\right]^{\lambda}}{g(w)} 
	\prec 
	\widetilde{p_{\kappa}(w)}
	=\dfrac{1+\tau^{2}_{\kappa}w^{2}}{1-\kappa\tau_{\kappa} w - \tau^{2}_{\kappa}w^{2}},\;w\in\mathbb{D},
	\end{equation}
	where
	$\tau_{\kappa} = \dfrac{\kappa-\sqrt{\kappa^{2}+4}}{2}.$
\end{definition}

\begin{enumerate}
	\item For $\lambda=1,$ we have the class $\mathcal{S}\mathcal{L}\mathcal{B}_{\Sigma}^{\kappa}(1;\tilde{p}_{\kappa})\equiv \mathcal{S}\mathcal{L}_{\Sigma}^{\kappa}(\tilde{p}_{\kappa}).$ A function $f\in
	\Sigma$ of the form 
	\begin{equation*}
	f(z) = z + \sum\limits_{n=2}^{\infty} a_{n} z^{n}, 
	\end{equation*}
	is said to be in $\mathcal{S}\mathcal{L}_{\Sigma}^{\kappa}(\tilde{p}_{\kappa}),$ if the following
	conditions 
	\begin{equation*}
	\frac{zf^{\prime}(z)}{f(z)} \prec \widetilde{p_{\kappa}(z)}=\dfrac{1+\tau^{2}_{\kappa}z^{2}}{1-\kappa\tau_{\kappa} z - \tau^{2}_{\kappa}z^{2}},\;z\in\mathbb{D} 
	\end{equation*}
	and for $g(w)=f^{-1}(w)$ 
	\begin{equation*}
	\frac{wg^{\prime}(w)}{g(w)} 
	\prec 
	\widetilde{p_{\kappa}(w)}
	=\dfrac{1+\tau^{2}_{\kappa}w^{2}}{1-\kappa\tau_{\kappa} w - \tau^{2}_{\kappa}w^{2}},\;w\in\mathbb{D},
	\end{equation*}
	hold, where
	$\tau_{\kappa} = \dfrac{\kappa-\sqrt{\kappa^{2}+4}}{2}.$
\end{enumerate}
%%----------------------Pasco-Type------------------------------------
\begin{definition}
	\label{def1.4} A function $f\in \Sigma$ of the form 
	\begin{equation*}
	f(z) = z + \sum\limits_{n=2}^{\infty} a_{n} z^{n}, 
	\end{equation*}
	belongs to the class $\mathcal{P}\mathcal{S}\mathcal{L}_{\Sigma}^{\kappa}(\lambda;\tilde{p}_{\kappa}),$  $0 \leq \lambda \leq 1,$ if the following conditions are satisfied: 
	\begin{equation}  \label{C1-e1}
	\frac{zf^{\prime}(z)+\lambda z^{2}f^{\prime\prime}(z)}{(1-\lambda)f(z)+\lambda zf^{\prime}(z)} 
	\prec 
	\widetilde{p_{\kappa}(z)}=\dfrac{1+\tau^{2}_{\kappa}z^{2}}{1-\kappa\tau_{\kappa} z - \tau^{2}_{\kappa}z^{2}},\;z\in\mathbb{D} 
	\end{equation}
	and for $g(w)=f^{-1}(w)$ 
	\begin{equation}  \label{C1-e2}
	\frac{wf^{\prime}(w)+\lambda w^{2}g^{\prime\prime}(w)}{(1-\lambda)g(w)+\lambda wg^{\prime}(w)} 
	\prec 
	\widetilde{p_{\kappa}(w)}
	=\dfrac{1+\tau^{2}_{\kappa}w^{2}}{1-\kappa\tau_{\kappa} w - \tau^{2}_{\kappa}w^{2}},\;w\in\mathbb{D},
	\end{equation}
	where
	$\tau_{\kappa} = \dfrac{\kappa-\sqrt{\kappa^{2}+4}}{2}.$
\end{definition}

\begin{enumerate}
	\item For $\lambda=0,$ we have the class $\mathcal{P}\mathcal{S}\mathcal{L}_{\Sigma}^{\kappa}(0;\tilde{p}_{\kappa})\equiv \mathcal{S}\mathcal{L}_{\Sigma}^{\kappa}(\tilde{p}_{\kappa}).$ 

	\item For $\lambda=1,$ we have the class $\mathcal{P}\mathcal{S}\mathcal{L}_{\Sigma}^{\kappa}(1;\tilde{p}_{\kappa})\equiv \mathcal{K}\mathcal{S}\mathcal{L}_{\Sigma}^{\kappa}(\tilde{p}_{\kappa}).$ A function $f\in
	\Sigma$ of the form 
	\begin{equation*}
	f(z) = z + \sum\limits_{n=2}^{\infty} a_{n} z^{n}, 
	\end{equation*}
	is said to be in $\mathcal{K}\mathcal{S}\mathcal{L}_{\Sigma}^{\kappa}(\tilde{p}_{\kappa}),$ if the following
	conditions 
	\begin{equation*}
	1 + \frac{z^{2}f^{\prime\prime}(z)}{f^{\prime}(z)} 
	\prec 
	\widetilde{p_{\kappa}(z)}=\dfrac{1+\tau^{2}_{\kappa}z^{2}}{1-\kappa\tau_{\kappa} z - \tau^{2}_{\kappa}z^{2}},\;z\in\mathbb{D} 
	\end{equation*}
	and for $g(w)=f^{-1}(w)$ 
	\begin{equation*}
	1 + \frac{w^{2}g^{\prime\prime}(w)}{g^{\prime}(w)} 
	\prec 
	\widetilde{p_{\kappa}(w)}
	=\dfrac{1+\tau^{2}_{\kappa}w^{2}}{1-\kappa\tau_{\kappa} w - \tau^{2}_{\kappa}w^{2}},\;w\in\mathbb{D},
	\end{equation*}
	hold, where
	$\tau_{\kappa} = \dfrac{\kappa-\sqrt{\kappa^{2}+4}}{2}.$
	\begin{remark}
		For $\gamma=0,$ $\mathcal{P}\mathcal{S}\mathcal{L}_{\Sigma}^{\kappa}(0, \tilde{p}_{\kappa})\equiv \mathcal{S}\mathcal{L}_{\Sigma}^{\kappa}(\tilde{p}_{\kappa})$ and $\gamma=1,$  $\mathcal{P}\mathcal{S}\mathcal{L}_{\Sigma}^{\kappa}(1, \tilde{p}_{\kappa})\equiv \mathcal{K}\mathcal{S}\mathcal{L}_{\Sigma}^{\kappa}(\tilde{p}_{\kappa})$ the classes were introduced and studied  G\"{u}ney et al. \cite{HOG-GMS-JS-2019-CFS-kF} and for $\kappa=1$ the class was studied by G\"{u}ney et al. \cite{HOG-GMS-JS-Fib-2018}.
	\end{remark}
\end{enumerate}

%%====================================================================================
In order to prove our results for the function in the classes $\mathcal{W}\mathcal{S}\mathcal{L}
_{\Sigma}^{\kappa}(\gamma, \lambda, \alpha, \tilde{p}_{\kappa}),$   $\mathcal{S}\mathcal{L}\mathcal{B}_{\Sigma}^{\kappa}(\lambda;\tilde{p}_{\kappa})$ and $\mathcal{P}\mathcal{S}\mathcal{L}_{\Sigma}^{\kappa}(\lambda;\tilde{p}_{\kappa})$, we need the following lemma. 

\begin{lemma}
	\cite{Pom}\label{lem-pom} If $p\in \mathcal{P},$ then $|p_i|\leqq 2$ for
	each $i,$ where $\mathcal{P}$ is the family of all functions $p,$ analytic
	in $\mathbb{D},$ for which
	\begin{equation*}
	\Re\{p(z)\}>0 \qquad z \in \mathbb{D},
	\end{equation*}
	where
	\begin{equation*}
	p(z)=1+p_1z+p_2z^2+\cdots \qquad z \in \mathbb{D}.
	\end{equation*}
\end{lemma}
%----CR-RAP-NM-P1-theorem-------------------
In this investigation, we find the estimates for the coefficients $|a_2|$
and $|a_3|$ for functions in the subclasses $\mathcal{W}\mathcal{S}\mathcal{L}_{\Sigma}^{\kappa}(\gamma,
\lambda, \alpha, \tilde{p}_{\kappa})$,  $\mathcal{R}\mathcal{S}\mathcal{L}_{\Sigma}^{\kappa}(\gamma, \lambda, \tilde{p}_{\kappa}),$ $\mathcal{S}\mathcal{L}\mathcal{B}_{\Sigma}^{\kappa}(\lambda;\tilde{p}_{\kappa})$ and $\mathcal{P}\mathcal{S}\mathcal{L}_{\Sigma}^{\kappa}(\lambda;\tilde{p}_{\kappa})$. Further, Fekete-Szeg\"{o} problems of the aforementioned classes. 
%%===========================================================================================================================================

\section{Initial Coefficient Estimates and Fekete-Szeg\"{o} Inequalities}

In the following theorem, we obtain coefficient estimates for functions in
the class $f\in\mathcal{W}\mathcal{S}\mathcal{L}_{\Sigma}^{\kappa}(\gamma, \lambda, \alpha, \tilde{p}_{\kappa}).$

\begin{theorem}
\label{Th1} Let $f(z)=z+\sum\limits_{n=2}^{\infty}a_{n}z^{n}$ be in the
class $\mathcal{W}\mathcal{S}\mathcal{L}_{\Sigma}^{\kappa}(\gamma, \lambda, \alpha, \tilde{p}_{\kappa})$. Then 
\begin{align*}
\left|a_{2}\right| &\leq \dfrac{\left|\gamma\right|\left|\kappa\tau_{\kappa}\right|\sqrt{\kappa}}{\sqrt{\gamma \kappa^{2}\tau_{\kappa}\left(1+2\alpha + 2\lambda\right)+(\kappa-(\kappa^{2}+2)\tau_{\kappa})(1+\alpha)^{2}}},
\end{align*}
\begin{align*}
\left|a_3\right|&\leq \dfrac{\left|\gamma\right|\left|\kappa\tau_{\kappa}\right|(\kappa-(\kappa^{2}+2)\tau_{\kappa})(1+\alpha)^{2}}{(1+2\alpha +2\lambda)\left[\gamma\kappa^{2}\tau_{\kappa}(1+2\alpha +2\lambda)+(\kappa-(\kappa^{2}+2)\tau_{\kappa})(1+\alpha)^{2}\right]}
\end{align*}
and
\begin{eqnarray*}
\left|a_{3}-\mu a_{2}^{2}\right|
\leq \left\{
\begin{array}{ll}
\dfrac{\gamma\left|\kappa\tau_{\kappa}\right|}{\left(1+2\alpha + 2\lambda\right)}; \\ 
\qquad\qquad\qquad 
0 \leq \left|\mu-1\right|\leq \dfrac{\gamma\kappa^{2}\tau_{\kappa}(1+2\alpha+2\lambda) + (\kappa-(\kappa^{2}+2)\tau_{\kappa})(1+\alpha)^{2}}{\gamma\kappa^{2}\tau_{\kappa}\left(1+2\alpha + 2\lambda\right)}\\
\dfrac{|1-\mu|\gamma^{2}\kappa^{3}\tau_{\kappa}^{2}}{{\gamma\kappa^{2}\tau_{\kappa}(1+2\alpha+2\lambda) + (\kappa-(\kappa^{2}+2)\tau_{\kappa})(1+\alpha)^{2}}}; \\
\qquad \qquad \qquad
\left|\mu-1\right|\geq \dfrac{\gamma\kappa^{2}\tau_{\kappa}(1+2\alpha+2\lambda) + (\kappa-(\kappa^{2}+2)\tau_{\kappa})(1+\alpha)^{2}}{\gamma\kappa^{2}\tau_{\kappa}\left(1+2\alpha + 2\lambda\right)}.
\end{array}		
\right.
\end{eqnarray*}
%where
%\begin{eqnarray*}
%	h(\mu) = \dfrac{\left(1-\mu\right)\gamma^{2}\kappa^{3}\tau^{2}_{\kappa}}{4\left[
%		\gamma\kappa^{2}\tau_{\kappa}(1+2\alpha+2\lambda) 
%		+ (\kappa-(\kappa^{2}+2)\tau_{\kappa})(1+\alpha)^{2}
%		\right]}.
%\end{eqnarray*}
\end{theorem}

\begin{proof}
Since $f\in\mathcal{W}\mathcal{S}\mathcal{L}_{\Sigma}^{\kappa}(\gamma, \lambda, \alpha, \tilde{p}_{\kappa})$, from the
Definition \ref{def1.1} we have 
\begin{equation}  \label{2.2}
1+\frac{1}{\gamma}\left((1-\alpha+2\lambda)\frac{f(z)}{z}+(\alpha-2%
\lambda)f^{\prime }(z)+\lambda zf^{\prime \prime }(z)-1\right) 
= \widetilde{p_{\kappa}(u(z))}
\end{equation}
and 
\begin{equation}  \label{2.3}
1+\frac{1}{\gamma}\left((1-\alpha+2\lambda)\frac{g(w)}{w}+(\alpha-2%
\lambda)g^{\prime }(w)+\lambda wg^{\prime \prime }(w)-1\right)
=\widetilde{p_{\kappa}(v(w))},
\end{equation}
where $z,\; w \in \mathbb{D}$ and $g=f^{-1}$. Using the fact the function $p(z)$ of the form  
\[
p(z) = 1+p_{1}z+p_{2}z^{2}+\ldots
\]
and 
$p\prec \tilde{p}_{\kappa}.$ Then there exists an analytic function $u$ such that 
$\left|u(z)\right|<1$ in $\mathbb{D}$ and $p(z)=\tilde{p}_{\kappa}(u(z)).$ Therefore,
define the function 
\begin{equation*}
h(z) = \dfrac{1+u(z)}{1-u(z)} = 1+c_{1}z+c_{2}z^{2}+\ldots
\end{equation*}
is in the class $\mathcal{P}.$ It follows that
\begin{equation*}
u(z) =\dfrac{h(z)-1}{h(z)+1} = \dfrac{c_{1}}{2}z + \left(c_{2}-\dfrac{c_{1}^{2}}{2}\right)\dfrac{z^{2}}{2}
+ \left(c_{3}-c_{1}c_{2}+\dfrac{c_{1}^{3}}{4}\right)\dfrac{z^{3}}{2}+\ldots
\end{equation*}
and 
\begin{eqnarray}
\tilde{p}_{\kappa}(u(z)) &=& 1+\tilde{p}_{\kappa,\; 1}\left(\dfrac{c_{1}}{2}z + \left(c_{2}-\dfrac{c_{1}^{2}}{2}\right)\dfrac{z^{2}}{2}
+ \left(c_{3}-c_{1}c_{2}+\dfrac{c_{1}^{3}}{4}\right)\dfrac{z^{3}}{2}+\ldots\right)\nonumber\\
&& \quad + \tilde{p}_{\kappa,\;2}\left(\dfrac{c_{1}}{2}z + \left(c_{2}-\dfrac{c_{1}^{2}}{2}\right)\dfrac{z^{2}}{2}
+ \left(c_{3}-c_{1}c_{2}+\dfrac{c_{1}^{3}}{4}\right)\dfrac{z^{3}}{2}+\ldots\right)^{2}\nonumber\\
&& \quad + \tilde{p}_{\kappa,\;3} \left(\dfrac{c_{1}}{2}z + \left(c_{2}-\dfrac{c_{1}^{2}}{2}\right)\dfrac{z^{2}}{2}
+ \left(c_{3}-c_{1}c_{2}+\dfrac{c_{1}^{3}}{4}\right)\dfrac{z^{3}}{2}+\ldots\right)^{3} + \ldots \nonumber\\
&=& 1+ \dfrac{\tilde{p}_{\kappa,\;1}c_{1}}{2}z 
+ \left(\dfrac{1}{2}\left(c_{2}-\dfrac{c_{1}^{2}}{2}\right)\tilde{p}_{\kappa,\;1}
+\dfrac{c_{1}^{2}}{4}\tilde{p}_{\kappa,\;2}\right)z^{2}\nonumber\\
&& \quad
+ \left(\dfrac{1}{2}\left(c_{3}-c_{1}c_{2}+\dfrac{c_{1}^{3}}{4}\right)\tilde{p}_{\kappa,\;1}
+\dfrac{1}{2}c_{1}\left(c_{2}-\dfrac{c_{1}^{2}}{2}\right)\tilde{p}_{\kappa,\;2}
+\dfrac{c_{1}^{3}}{8}\tilde{p}_{\kappa,\;3}\right)z^{3} \label{2.4}\\
&& \quad + \ldots. \nonumber
\end{eqnarray}
Similarly, there exists an analytic function $v$ such that $\left|v(w)\right|<1$ in $\mathbb{D}$ and $p(w)=\tilde{p}(v(w)).$ Therefore, the function
\begin{equation*}
k(w) = \dfrac{1+v(w)}{1-v(w)} = 1+d_{1}w+d_{2}w^{2}+\ldots
\end{equation*}
is in the class $\mathcal{P}.$ It follows that
\begin{equation*}
v(w) =\dfrac{k(w)-1}{k(w)+1} = \dfrac{d_{1}}{2}w + \left(d_{2}-\dfrac{d_{1}^{2}}{2}\right)\dfrac{w^{2}}{2}
+ \left(d_{3}-d_{1}d_{2}+\dfrac{d_{1}^{3}}{4}\right)\dfrac{w^{3}}{2}+\ldots
\end{equation*}
and 
\begin{eqnarray}
\tilde{p}_{\kappa}(v(w)) &=& 1+\tilde{p}_{\kappa,\;1}\left(\dfrac{d_{1}}{2}w + \left(d_{2}-\dfrac{d_{1}^{2}}{2}\right)\dfrac{w^{2}}{2}
+ \left(d_{3}-d_{1}d_{2}+\dfrac{d_{1}^{3}}{4}\right)\dfrac{w^{3}}{2}+\ldots\right)\nonumber\\
&& \quad + \tilde{p}_{\kappa,\;2}\left(\dfrac{d_{1}}{2}w + \left(d_{2}-\dfrac{d_{1}^{2}}{2}\right)\dfrac{w^{2}}{2}
+ \left(d_{3}-d_{1}d_{2}+\dfrac{d_{1}^{3}}{4}\right)\dfrac{w^{3}}{2}+\ldots\right)^{2}\nonumber\\
&& \quad + \tilde{p}_{\kappa,\;3} \left(\dfrac{d_{1}}{2}w + \left(d_{2}-\dfrac{d_{1}^{2}}{2}\right)\dfrac{w^{2}}{2}
+ \left(d_{3}-d_{1}d_{2}+\dfrac{d_{1}^{3}}{4}\right)\dfrac{w^{3}}{2}+\ldots\right)^{3} + \ldots \nonumber\\
&=& 1+ \dfrac{\tilde{p}_{\kappa,\;1}d_{1}}{2}w 
+ \left(\dfrac{1}{2}\left(d_{2}-\dfrac{d_{1}^{2}}{2}\right)\tilde{p}_{\kappa,\;1}
+\dfrac{d_{1}^{2}}{4}\tilde{p}_{\kappa,\;2}\right)w^{2}\nonumber\\
&& \quad
+ \left(\dfrac{1}{2}\left(d_{3}-d_{1}d_{2}+\dfrac{d_{1}^{3}}{4}\right)\tilde{p}_{\kappa,\;1}
+\dfrac{1}{2}d_{1}\left(d_{2}-\dfrac{d_{1}^{2}}{2}\right)\tilde{p}_{\kappa,\;2}
+\dfrac{d_{1}^{3}}{8}\tilde{p}_{\kappa,\;3}\right)w^{3} \label{2.5}\\
&& \quad + \ldots. \nonumber
\end{eqnarray}
%%-----------------------------
%that the functions $p$ have the following Taylor expansions 
%\begin{eqnarray}
%&&\widetilde{p(u(z))}=1+\dfrac{\tilde{p}_1c_{1}}{2}z+\left[\dfrac{1}{2}\left(c_{2}-\d%frac{c_{1}^{2}}{2}\right)\tilde{p_{1}}+\dfrac{c_{1}^{2}}{4}\tilde{p_{2}}\right]z^2\no%number\\
%&&\quad %+\left[\dfrac{1}{2}\left(c_{3}-c_{1}c_{2}+\dfrac{c_{1}^{3}}{4}\right)\tilde{p_1}+\dfr%ac{1}{2}c_{1}\left(c_{2}-\dfrac{c_{1}^{2}}{2}\right)\tilde{p_{2}}+\dfrac{c_{1}^{3}}{8%}\tilde{p_{3}}\right]z^3+\dots,\;z\in\mathbb{D},  \label{2.4} \\
%&&\widetilde{p(v(w))}=1+\dfrac{\tilde{p}_1d_{1}}{2}w+\left[\dfrac{1}{2}\left(d_{2}-\d%frac{d_{1}^{2}}{2}\right)\tilde{p_{1}}+\dfrac{d_{1}^{2}}{4}\tilde{p_{2}}\right]w^2\no%number\\
%&&\quad %+\left[\dfrac{1}{2}\left(d_{3}-d_{1}d_{2}+\dfrac{d_{1}^{3}}{4}\right)\tilde{p_1}+\dfr%ac{1}{2}d_{1}\left(d_{2}-\dfrac{d_{1}^{2}}{2}\right)\tilde{p_{2}}+\dfrac{d_{1}^{3}}{8%}\tilde{p_{3}}\right]w^3+\dots,\;w\in\mathbb{D},  \label{2.5}
%\end{eqnarray}
By virtue of \eqref{2.2}, \eqref{2.3}, \eqref{2.4} and \eqref{2.5}, we have 
\begin{equation}  \label{2.6}
\frac{1}{\gamma}(1+\alpha)a_2=\dfrac{c_{1}\kappa\tau_{\kappa}}{2},
\end{equation}
\begin{equation}  \label{2.7}
\frac{a_3}{\gamma}(1+2\alpha+2\lambda)=\dfrac{1}{2}\left(c_{2}-\dfrac{c_{1}^{2}}{2}\right)\kappa\tau_{\kappa}+\dfrac{c_{1}^{2}}{4}(\kappa^{2}+2)\tau^{2}_{\kappa},
\end{equation}

\begin{equation}  \label{2.8}
-\frac{1}{\gamma}(1+\alpha)a_2=\dfrac{d_{1}\kappa\tau_{\kappa}}{2},
\end{equation}

and 
\begin{equation}  \label{2.9}
\frac{(1+2\alpha+2\lambda)}{\gamma}(2a^2_2-a_3) =\dfrac{1}{2}\left(d_{2}-\dfrac{d_{1}^{2}}{2}\right)\kappa\tau_{\kappa}+\dfrac{d_{1}^{2}}{4}(\kappa^{2}+2)\tau^{2}_{\kappa}.
\end{equation}

From \eqref{2.6} and \eqref{2.8}, we obtain
\begin{equation*} 
c_{1}=-d_{1},
\end{equation*}
and
\begin{eqnarray} 
\dfrac{2}{\gamma^{2}}(1+\alpha)^{2} a_{2}^2 &=& \dfrac{(c_{1}^{2} + d_{1}^{2})\kappa^{2}\tau^{2}_{\kappa}}{4}\nonumber\\
a_{2}^2 &=& \dfrac{\gamma^{2}(c_{1}^{2} + d_{1}^{2})\kappa^{2}\tau^{2}_{\kappa}}{8(1+\alpha)^{2}}~.\label{2.10}
\end{eqnarray}
By adding  \eqref{2.7} and \eqref{2.9}, we have 
\begin{eqnarray}
\dfrac{2}{\gamma}\left(1+2\alpha + 2\lambda\right) a_{2}^{2}
= \dfrac{1}{2}\left(c_{2}+d_{2}\right)\kappa\tau_{\kappa} - \dfrac{1}{4}\left(c_{1}^{2}+d_{1}^{2}\right)\kappa\tau_{\kappa} + \dfrac{1}{4}\left(c_{1}^{2}+d_{1}^{2}\right)(\kappa^{2}+2)\tau^{2}_{\kappa}.\label{2.11}
\end{eqnarray}
By substituting  \eqref{2.10} in \eqref{2.11}, we reduce that
\begin{eqnarray}
a_{2}^{2} &=& \dfrac{\gamma^{2}\left(c_{2}+d_{2}\right)\kappa^{3}\tau^{2}_{\kappa}}{4\left[\gamma \kappa^{2}\tau_{\kappa}\left(1+2\alpha + 2\lambda\right)+(\kappa-(\kappa^{2}+2)\tau_{\kappa})(1+\alpha)^{2}\right]}.\label{2.12}
\end{eqnarray}
Now, applying Lemma \ref{lem-pom}, we obtain
\begin{eqnarray}
\left|a_{2}\right| &\leq \dfrac{\left|\gamma\right|\left|\kappa\tau_{\kappa}\right|\sqrt{\kappa}}{\sqrt{\gamma \kappa^{2}\tau_{\kappa}\left(1+2\alpha + 2\lambda\right)+(\kappa-(\kappa^{2}+2)\tau_{\kappa})(1+\alpha)^{2}}}.\label{2.13}
\end{eqnarray}
By subtracting \eqref{2.9} from \eqref{2.7}, we obtain
\begin{eqnarray}
a_{3}= \dfrac{\gamma\left(c_{2}-d_{2}\right)\kappa\tau_{\kappa}}{4\left(1+2\alpha + 2\lambda\right)}  + a_{2}^{2}.\label{2.14}
\end{eqnarray}
Hence by Lemma \ref{lem-pom}, we have 
\begin{eqnarray}
\left|a_{3}\right|&\leq& \dfrac{\left|\gamma\right|\left(\left|c_{2}\right|+\left|d_{2}\right|\right)\left|\kappa\tau_{\kappa}\right|}{4\left(1+2\alpha + 2\lambda\right)}  + \left|a_{2}\right|^{2}\leq  \dfrac{\left|\gamma\right|\left|\kappa\tau_{\kappa}\right|}{\left(1+2\alpha + 2\lambda\right)} + \left|a_{2}\right|^{2}
.\label{2.15}
\end{eqnarray}
Then in view of \eqref{2.13}, we obtain
\begin{eqnarray*}
\left|a_{3}\right|
\leq 
\dfrac{\left|\gamma\right|\left|\kappa\tau_{\kappa}\right|(\kappa-(\kappa^{2}+2)\tau_{\kappa})(1+\alpha)^{2}}{(1+2\alpha +2\lambda)\left[\gamma\kappa^{2}\tau_{\kappa}(1+2\alpha +2\lambda)+(\kappa-(\kappa^{2}+2)\tau_{\kappa})(1+\alpha)^{2}\right]}
\end{eqnarray*}
	From \eqref{2.14}, we have 
	\begin{eqnarray}\label{Th1-Fekete-e1}
	a_{3} -\mu a_{2}^{2} = \dfrac{\gamma\left(c_{2}-d_{2}\right)\kappa\tau_{\kappa}}{4\left(1+2\alpha + 2\lambda\right)} + \left(1-\mu\right)a_{2}^{2}.
	\end{eqnarray}
	By substituting \eqref{2.12} in \eqref{Th1-Fekete-e1}, we have 
	\begin{eqnarray}\label{Th1-Fekete-e2}
	a_{3} -\mu a_{2}^{2} 
	&=& \dfrac{\gamma\left(c_{2}-d_{2}\right)\kappa\tau_{\kappa}}{4\left(1+2\alpha + 2\lambda\right)} + \left(1-\mu\right) \left(\dfrac{\gamma^{2}\left(c_{2}+d_{2}\right)\kappa^{3}\tau^{2}_{\kappa}}{4\left[\gamma \kappa^{2}\tau_{\kappa}\left(1+2\alpha + 2\lambda\right)+(\kappa-(\kappa^{2}+2)\tau_{\kappa})(1+\alpha)^{2}\right]}\right)\nonumber\\
	&=& \left(h(\mu) + \dfrac{\gamma\kappa\left|\tau_{\kappa}\right|}{4\left(1+2\alpha + 2\lambda\right)}\right)c_{2}
		+ \left(h(\mu) - \dfrac{\gamma\kappa\left|\tau_{\kappa}\right|}{4\left(1+2\alpha + 2\lambda\right)}\right)d_{2},
	\end{eqnarray}
	where
	\begin{eqnarray*}
	h(\mu) =\dfrac{\left(1-\mu\right)\gamma^{2}\kappa^{3}\tau^{2}_{\kappa}}{4\left[
			\gamma\kappa^{2}\tau_{\kappa}(1+2\alpha+2\lambda) 
			+ (\kappa-(\kappa^{2}+2)\tau_{\kappa})(1+\alpha)^{2}
			\right]}.	
	\end{eqnarray*}
	Thus by taking modulus of \eqref{Th1-Fekete-e2}, we conclude that
	\begin{eqnarray*}
	\left|a_{3}-\mu a_{2}^{2}\right|
	\leq \left\{
	\begin{array}{ll}
	\dfrac{|\gamma|\left|\kappa\tau_{\kappa}\right|}{\left(1+2\alpha + 2\lambda\right)} &; 
	0 \leq \left|h(\mu)\right|\leq \dfrac{|\gamma|\left|\kappa\tau_{\kappa}\right|}{4\left(1+2\alpha + 2\lambda\right)}\\
	4\left|h(\mu)\right| &; 
	\left|h(\mu)\right|\geq \dfrac{|\gamma|\left|\kappa\tau_{\kappa}\right|}{4\left(1+2\alpha + 2\lambda\right)}.
	\end{array}		
	\right.
	\end{eqnarray*}
\end{proof}
%%================================================================================
In the following theorem, we obtain coefficient estimates for functions in
the class $f\in\mathcal{R}\mathcal{S}\mathcal{L}_{\Sigma}^{\kappa}(\gamma, \lambda, \tilde{p}_{\kappa})$.

\begin{theorem}
	\label{thm3.1} Let $f(z)=z+\sum\limits_{n=2}^{\infty}a_{n}z^{n}$ be in the
	class $\mathcal{R}\mathcal{S}\mathcal{L}_{\Sigma}{\kappa}(\gamma, \lambda, \tilde{p}_{\kappa})$. Then 
	\begin{align*}
	\left|a_2\right|&\leq 
	\dfrac{\sqrt{2}\left|\gamma\right|\left|\kappa\tau_{\kappa}\right|\sqrt{\kappa}}{\sqrt{\gamma \kappa^{2}\tau_{\kappa}\left(2+\lambda\right)\left(1+\lambda\right)+2(\kappa - (\kappa^{2}+2)\tau_{\kappa})(1+\lambda)^{2}}},
	\end{align*}
	\begin{align*}
	\left|a_3\right|&\leq \dfrac{\left|\gamma\right|\left|\kappa\tau_{\kappa}\right|\left\{\gamma \kappa^{2}\tau_{\kappa} \left(2+\lambda\right)\left(1+\lambda\right)+2(\kappa - (\kappa^{2}+2)\tau_{\kappa})(1+\lambda)^{2}-2\left(2+\lambda\right)\gamma\kappa^{2}\tau_{\kappa}\right\}}{(2+\lambda)\left[\gamma\kappa^{2}\tau_{\kappa}(2+\lambda)\left(1+\lambda\right)+2(\kappa - (\kappa^{2}+2)\tau_{\kappa})(1+\lambda)^{2}\right]}
	\end{align*}
	and
	\begin{eqnarray*}
		\left|a_{3}-\mu a_{2}^{2}\right|
		\leq \left\{
		\begin{array}{ll}
			\dfrac{\left|\gamma\right|\left|\kappa\tau_{\kappa}\right|}{2+\lambda} &; 
			0 \leq \left|\mu-1\right|\leq \dfrac{M}{2\left|\gamma\right|\kappa^{2}\left|\tau_{\kappa}\right|\left(2+\lambda\right)}\\
			\dfrac{2\left|1-\mu\right|\gamma^{2}\kappa^{3}\tau_{\kappa}^{2}}{M} &; 
			\left|\mu-1\right|\geq \dfrac{M}{2\left|\gamma\right|\kappa^{2}\left|\tau_{\kappa}\right|\left(2+\lambda\right)},
		\end{array}		
		\right.
	\end{eqnarray*}
	where
	\begin{eqnarray*}
		M = \gamma\kappa^{2}\tau_{\kappa}\left(2+\lambda\right)\left(1+\lambda\right)+2\left(1+\lambda\right)^{2}\left(\kappa-(\kappa^{2}+2)\tau_{\kappa}\right).
	\end{eqnarray*}
\end{theorem}

\begin{proof}
	Since $f\in\mathcal{R}\mathcal{S}\mathcal{L}_{\Sigma}^{\kappa}(\gamma, \lambda, \tilde{p}_{\kappa})$, from the
	Definition \ref{def1.2} we have 
	\begin{equation}  \label{3.2}
	1+\frac{1}{\gamma}\left(\frac{z^{1-\lambda}f^{\prime }(z)}{(f(z))^{1-\lambda}%
	}-1\right) = \widetilde{p_{\kappa}(u(z))}
	\end{equation}
	and 
	\begin{equation}  \label{3.3}
	1+\frac{1}{\gamma}\left(\frac{w^{1-\lambda}g^{\prime }(w)}{(g(w))^{1-\lambda}%
	}-1\right) =\widetilde{p_{\kappa}(v(w))}.
	\end{equation}
	By virtue of \eqref{3.2}, \eqref{3.3}, \eqref{2.4} and \eqref{2.5}, we get 
	\begin{equation}  \label{3.6}
	\frac{1}{\gamma}(1+\lambda)a_2=\dfrac{c_{1}\kappa\tau_{\kappa}}{2},
	\end{equation}
	\begin{equation}  \label{3.7}
	\frac{1}{\gamma}(2+\lambda)\left[a_{3}+
	\left(\lambda-1\right)\dfrac{a_{2}^{2}}{2}\right]
	=\dfrac{1}{2}\left(c_{2}-\dfrac{c_{1}^{2}}{2}\right)\kappa\tau_{\kappa}
	+\dfrac{c_{1}^{2}}{4}(\kappa^{2}+2)\tau_{\kappa}^{2},
	\end{equation}
	\begin{equation}  \label{3.8}
	-\frac{1}{\gamma}(1+\lambda)a_2=\dfrac{d_{1}\kappa\tau_{\kappa}}{2}
	\end{equation}
	and 
	\begin{equation}  \label{3.9}
	\frac{1}{\gamma}(2+\lambda)\left[\left(3+\lambda\right)\dfrac{a_{2}^{2}}{2}-a_{3}\right] =\dfrac{1}{2}\left(d_{2}-\dfrac{d_{1}^{2}}{2}\right)\kappa\tau_{\kappa}
	+\dfrac{d_{1}^{2}}{4}(\kappa^{2}+2)\tau_{\kappa}^{2}.
	\end{equation}
	From \eqref{3.6} and \eqref{3.8}, we obtain
	\begin{equation*} 
	c_{1}=-d_{1},
	\end{equation*}
	and
	\begin{eqnarray} 
	\dfrac{2}{\gamma^{2}}(1+\lambda)^{2} a_{2}^2 &=& \dfrac{(c_{1}^{2} + d_{1}^{2})\kappa^{2}\tau_{\kappa}^{2}}{4}\nonumber\\
	a_{2}^2 &=& \dfrac{\gamma^{2}(c_{1}^{2} + d_{1}^{2})\kappa^{2}\tau_{\kappa}^{2}}{8(1+\lambda)^{2}}~.\label{3.10}
	\end{eqnarray}
	By adding  \eqref{3.7} and \eqref{3.9}, we have 
	\begin{eqnarray}
	\dfrac{1}{\gamma}\left(2+\lambda\right)\left(1+\lambda\right) a_{2}^{2}
	= \dfrac{1}{2}\left(c_{2}+d_{2}\right)\kappa\tau_{\kappa} - \dfrac{1}{4}\left(c_{1}^{2}+d_{1}^{2}\right)\kappa\tau_{\kappa} + \dfrac{1}{4}\left(c_{1}^{2}+d_{1}^{2}\right)(\kappa^{2}+2)\tau_{\kappa}^{2}.\label{3.11}
	\end{eqnarray}
	By substituting  \eqref{3.10} in \eqref{3.11}, we reduce that
	\begin{eqnarray}
	a_{2}^{2} &=& \dfrac{\gamma^{2}\left(c_{2}+d_{2}\right)\kappa^{3}\tau_{\kappa}^{2}}{2\left[\gamma \kappa^{2}\tau_{\kappa}\left(2+\lambda\right)\left(1+\lambda\right)+2(\kappa - (\kappa^{2}+2)\tau_{\kappa})(1+\lambda)^{2}\right]}.\label{3.12}
	\end{eqnarray}
	Now, applying Lemma \ref{lem-pom}, we obtain
	\begin{eqnarray}
	\left|a_{2}\right| &\leq& \dfrac{\sqrt{2}\left|\gamma\right|\left|\kappa\tau_{\kappa}\right|\sqrt{\kappa}}{\sqrt{\gamma \kappa^{2}\tau_{\kappa}\left(2+\lambda\right)\left(1+\lambda\right)+2(\kappa - (\kappa^{2}+2)\tau_{\kappa})(1+\lambda)^{2}}}.\label{3.13}
	\end{eqnarray}
	By subtracting \eqref{3.9} from \eqref{3.7}, we obtain
	\begin{eqnarray}
	a_{3}= \dfrac{\gamma\left(c_{2}-d_{2}\right)\kappa\tau_{\kappa}}{4\left(2+\lambda\right)}  + a_{2}^{2}.\label{3.14}
	\end{eqnarray}
	Hence by Lemma \ref{lem-pom}, we have 
	\begin{eqnarray}
	\left|a_{3}\right|&\leq& \dfrac{\left|\gamma\right|\left(\left|c_{2}\right|+\left|d_{2}\right|\right)\left|\kappa\tau_{\kappa}\right|}{4\left(2+\lambda\right)}  + \left|a_{2}\right|^{2}\leq  \dfrac{\left|\gamma\right|\left|\kappa\tau_{\kappa}\right|}{\left(2+\lambda\right)} + \left|a_{2}\right|^{2}
	.\label{3.15}
	\end{eqnarray}
	Then in view of \eqref{3.13}, we obtain
	\begin{eqnarray*}
		\left|a_{3}\right|
		\leq 
		\dfrac{\left|\gamma\right|\left|\kappa\tau_{\kappa}\right|\left\{\gamma \kappa^{2}\tau_{\kappa} \left(2+\lambda\right)\left(1+\lambda\right)+2(\kappa - (\kappa^{2}+2)\tau_{\kappa})(1+\lambda)^{2}-2\left(2+\lambda\right)\gamma\kappa^{2}\tau_{\kappa}\right\}}{(2+\lambda)\left[\gamma\kappa^{2}\tau_{\kappa}(2+\lambda)\left(1+\lambda\right)+2(\kappa - (\kappa^{2}+2)\tau_{\kappa})(1+\lambda)^{2}\right]}.
	\end{eqnarray*}
	From \eqref{3.14}, we have 
	\begin{eqnarray}\label{Th2-Fekete-e1}
	a_{3} -\mu a_{2}^{2} = \dfrac{\gamma\left(c_{2}-d_{2}\right)\kappa\tau_{\kappa}}{4\left(2+\lambda\right)} + \left(1-\mu\right)a_{2}^{2}.
	\end{eqnarray}
	By substituting \eqref{3.12} in \eqref{Th2-Fekete-e1}, we have 
	\begin{eqnarray}\label{Th2-Fekete-e2}
	a_{3} -\mu a_{2}^{2} 
	&=& \dfrac{\gamma\left(c_{2}-d_{2}\right)\kappa\tau_{\kappa}}{4\left(2+\lambda\right)} + \left(1-\mu\right) \left(\dfrac{\gamma^{2}\left(c_{2}+d_{2}\right)\kappa^{3}\tau_{\kappa}^{2}}{2\left[\gamma\kappa^{2}\tau_{\kappa}\left(2+\lambda\right)\left(1+\lambda\right)+2\left(1+\lambda\right)^{2}\left(\kappa-(\kappa^{2}+2)\tau_{\kappa}\right)\right]}\right)\nonumber\\
	&=& \left(h(\mu) + \dfrac{\gamma\kappa\tau_{\kappa}}{4\left(2+\lambda\right)}\right)c_{2}
	+ \left(h(\mu) - \dfrac{\gamma\kappa\tau_{\kappa}}{4\left(2+\lambda\right)}\right)d_{2},
	\end{eqnarray}
	where
	\begin{eqnarray*}
		h(\mu) = \dfrac{\left(1-\mu\right)\gamma^{2}\kappa^{3}\tau_{\kappa}^{2}}{2\left[\gamma\kappa^{2}\tau_{\kappa}\left(2+\lambda\right)\left(1+\lambda\right)+2\left(1+\lambda\right)^{2}\left(\kappa-(\kappa^{2}+2)\tau_{\kappa}\right)\right]}.
	\end{eqnarray*}
	Thus by taking modulus of \eqref{Th2-Fekete-e2}, we conclude that
	\begin{eqnarray*}
		\left|a_{3}-\mu a_{2}^{2}\right|
		\leq \left\{
		\begin{array}{ll}
			\dfrac{\left|\gamma\right|\left|\kappa\tau_{\kappa}\right|}{\left(2+\lambda\right)} &; 
			0 \leq \left|h(\mu)\right|\leq \dfrac{\left|\gamma\right|\left|\kappa\tau_{\kappa}\right|}{4\left(2+\lambda\right)}\\
			4\left|h(\mu)\right| &; 
			\left|h(\mu)\right|\geq \dfrac{\left|\gamma\right|\left|\kappa\tau_{\kappa}\right|}{4\left(2+\lambda\right)}.
		\end{array}		
		\right.
	\end{eqnarray*}
	This gives desired inequality. 
\end{proof}
%%================================================================================
\begin{theorem}
	\label{thm4.1} Let $f(z)=z+\sum\limits_{n=2}^{\infty}a_{n}z^{n}$ be in the
	class $\mathcal{S}\mathcal{L}\mathcal{B}_{\Sigma}^{\kappa}(\lambda;\tilde{p}_{\kappa})$. Then 
	\begin{align*}
	\left|a_2\right|&\leq \dfrac{\left|\kappa\tau_{\kappa}\right|}{\sqrt{\left[\lambda\left(2\lambda-1\right)\kappa^{2}\tau_{\kappa} 
			+ \left(\kappa -  (\kappa^{2}+2)\tau_{\kappa}\right)
			\left(2\lambda-1\right)^{2}\right]}},
	\end{align*}
	\begin{align*}
	\left|a_3\right|&\leq \dfrac{\left|\kappa\tau_{\kappa}\right|\left[ \left(\kappa -  (\kappa^{2}+2)\tau_{\kappa}\right)(2\lambda-1)^{2}+\left(2\lambda^{2}-4\lambda+1\right)\kappa^{2}\tau_{\kappa}\right\}}{(3\lambda-1)\left[\lambda\left(2\lambda-1\right)\kappa^{2}\tau_{\kappa} 
		+ \left(\kappa -  (\kappa^{2}+2)\tau_{\kappa}\right)
		\left(2\lambda-1\right)^{2}\right]}
	\end{align*}
	and
	\begin{eqnarray*}
		\left|a_{3}-\mu a_{2}^{2}\right|
		\leq \left\{
		\begin{array}{ll}
			\dfrac{\left|\kappa\tau_{\kappa}\right|}{3\lambda-1};
			\\ \qquad\qquad\qquad 
			0 \leq \left|\mu-1\right|\leq \dfrac{\left[\lambda\left(2\lambda-1\right)\kappa^{2}\tau_{\kappa} 
				+ \left(\kappa -  (\kappa^{2}+2)\tau_{\kappa}\right)
				\left(2\lambda-1\right)^{2}\right]}{\kappa^{2}\tau_{\kappa}\left(3\lambda-1\right)}\\
			\dfrac{\left|1-\mu\right|\kappa^{3}\tau_{\kappa}^{2}}{\left[\lambda\left(2\lambda-1\right)\kappa^{2}\tau_{\kappa} 
				+ \left(\kappa -  (\kappa^{2}+2)\tau_{\kappa}\right)
				\left(2\lambda-1\right)^{2}\right]};
			\\\qquad\qquad\qquad
			\left|\mu-1\right|\geq \dfrac{\left[\lambda\left(2\lambda-1\right)\kappa^{2}\tau_{\kappa} 
				+ \left(\kappa -  (\kappa^{2}+2)\tau_{\kappa}\right)
				\left(2\lambda-1\right)^{2}\right]}{\left|\kappa^{3}\tau_{\kappa}\right|\left(3\lambda-1\right)}.
		\end{array}		
		\right.
	\end{eqnarray*}
	%	where
	%	\begin{eqnarray*}
	%		M = \left[\lambda\left(2\lambda-1\right)\kappa^{2}\tau_{\kappa} + %\left(\kappa - %(\kappa^{2}+2)\tau_{\kappa}\right)\left(2\lambda-1\right)^{2}\right].
	%	\end{eqnarray*}
\end{theorem}

\begin{proof}
	Since $f\in\mathcal{S}\mathcal{L}\mathcal{B}_{\Sigma}^{\kappa}(\lambda;\tilde{p}_{\kappa})$, from the
	Definition \ref{def1.3} we have 
	\begin{equation}  \label{4.2}
	\frac{z\left[f^{\prime}(z)\right]^{\lambda}}{f(z)}
	= \widetilde{p_{\kappa}(u(z))}
	\end{equation}
	and 
	\begin{equation}  \label{4.3}
	\frac{w\left[g^{\prime}(w)\right]^{\lambda}}{g(w)}
	=\widetilde{p_{\kappa}(v(w))}.
	\end{equation}
	By virtue of \eqref{4.2}, \eqref{4.3}, \eqref{2.4} and \eqref{2.5}, we get 
	\begin{equation}  \label{4.6}
	(2\lambda-1)a_2=\dfrac{c_{1}\kappa\tau_{\kappa}}{2},
	\end{equation}
	\begin{equation}  \label{4.7}
	(3\lambda-1)a_{3}+(2\lambda^{2}-4\lambda+1)a_{2}^{2}=\dfrac{1}{2}\left(c_{2}-\dfrac{c_{1}^{2}}{2}\right)\kappa\tau_{\kappa}+\dfrac{c_{1}^{2}}{4}(\kappa^{2}+2)\tau_{\kappa}^{2},
	\end{equation}
	
	\begin{equation}  \label{4.8}
	-(2\lambda-1)a_2=\dfrac{d_{1}\kappa\tau_{\kappa}}{2}
	\end{equation}
	and 
	\begin{equation}  \label{4.9}
	(2\lambda^{2}+2\lambda-1)a_{2}^{2} - (3\lambda-1)a_{3} =\dfrac{1}{2}\left(d_{2}-\dfrac{d_{1}^{2}}{2}\right)\kappa\tau_{\kappa}+\dfrac{d_{1}^{2}}{4}(\kappa^{2}+2)\tau_{\kappa}^{2}.
	\end{equation}
	
	From \eqref{4.6} and \eqref{4.8}, we obtain
	\begin{equation*} 
	c_{1}=-d_{1},
	\end{equation*}
	and
	\begin{eqnarray} 
	2(2\lambda-1)^{2} a_{2}^{2} &=& \dfrac{(c_{1}^{2} + d_{1}^{2})\kappa^{2}\tau_{\kappa}^{2}}{4}\nonumber\\
	a_{2}^{2} &=& \dfrac{(c_{1}^{2} + d_{1}^{2})\kappa^{2}\tau_{\kappa}^{2}}{8(2\lambda-1)^{2}}~.\label{4.10}
	\end{eqnarray}
	By adding  \eqref{4.7} and \eqref{4.9}, we have 
	\begin{eqnarray}
	2\lambda\left(2\lambda-1\right) a_{2}^{2}
	= \dfrac{1}{2}\left(c_{2}+d_{2}\right)\kappa\tau_{\kappa} - \dfrac{1}{4}\left(c_{1}^{2}+d_{1}^{2}\right)\kappa\tau_{\kappa} + \dfrac{1}{4}\left(c_{1}^{2}+d_{1}^{2}\right)(\kappa^{2}+2)\tau_{\kappa}^{2}.\label{4.11}
	\end{eqnarray}
	By substituting  \eqref{4.10} in \eqref{4.11}, we reduce that
	\begin{eqnarray}
	a_{2}^{2} &=& \dfrac{\left(c_{2}+d_{2}\right)\kappa^{3}\tau_{\kappa}^{2}}{4\left[\lambda\left(2\lambda-1\right)\kappa^{2}\tau_{\kappa} + \left(\kappa - (\kappa^{2}+2)\tau_{\kappa}\right)\left(2\lambda-1\right)^{2}\right]}.\label{4.12}
	\end{eqnarray}
	Now, applying Lemma \ref{lem-pom}, we obtain
	\begin{eqnarray}
	\left|a_{2}\right| &\leq& \dfrac{\left|\kappa\tau_{\kappa}\right|\sqrt{\kappa}}{\sqrt{\left[\lambda\left(2\lambda-1\right)\kappa^{2}\tau_{\kappa} + \left(\kappa - (\kappa^{2}+2)\tau_{\kappa}\right)\left(2\lambda-1\right)^{2}\right]}}.\label{4.13}
	\end{eqnarray}
	By subtracting \eqref{4.9} from \eqref{4.7}, we obtain
	\begin{eqnarray}
	a_{3}= \dfrac{\left(c_{2}-d_{2}\right)\kappa\tau_{\kappa}}{4\left(3\lambda-1\right)}  + a_{2}^{2}.\label{4.14}
	\end{eqnarray}
	Hence by Lemma \ref{lem-pom}, we have 
	\begin{eqnarray}
	\left|a_{3}\right|&\leq& \dfrac{\left(\left|c_{2}\right|+\left|d_{2}\right|\right)\left|\kappa\tau_{\kappa}\right|}{4\left(3\lambda-1\right)}  + \left|a_{2}\right|^{2}\leq  \dfrac{\left|\kappa\tau_{\kappa}\right|}{3\lambda-1} + \left|a_{2}\right|^{2}
	.\label{4.15}
	\end{eqnarray}
	Then in view of \eqref{4.13}, we obtain
	\begin{eqnarray*}
		\left|a_{3}\right|
		\leq 
		\dfrac{\left|\kappa\tau_{\kappa}\right|\left[ \left(\kappa -  (\kappa^{2}+2)\tau_{\kappa}\right)(2\lambda-1)^{2}+\left(2\lambda^{2}-4\lambda+1\right)\kappa^{2}\tau_{\kappa}\right\}}{(3\lambda-1)\left[\lambda\left(2\lambda-1\right)\kappa^{2}\tau_{\kappa} 
			+ \left(\kappa -  (\kappa^{2}+2)\tau_{\kappa}\right)
			\left(2\lambda-1\right)^{2}\right]}.
	\end{eqnarray*}
	From \eqref{4.14}, we have 
	\begin{eqnarray}\label{Th3-Fekete-e1}
	a_{3} -\mu a_{2}^{2} = \dfrac{\left(c_{2}-d_{2}\right)\kappa\tau_{\kappa}}{4\left(3\lambda-1\right)} + \left(1-\mu\right)a_{2}^{2}.
	\end{eqnarray}
	By substituting \eqref{4.12} in \eqref{Th3-Fekete-e1}, we have 
	\begin{eqnarray}\label{Th3-Fekete-e2}
	a_{3} -\mu a_{2}^{2} 
	&=& \dfrac{\left(c_{2}-d_{2}\right)\kappa\tau_{\kappa}}{4\left(3\lambda-1\right)} + \left(1-\mu\right) \left(\dfrac{\left(c_{2}+d_{2}\right)\kappa^{3}\tau_{\kappa}^{2}}{4\left[\lambda\left(2\lambda-1\right)\kappa^{2}\tau_{\kappa} + \left(\kappa - (\kappa^{2}+2)\tau_{\kappa}\right)\left(2\lambda-1\right)^{2}\right]}\right)\nonumber\\
	&=& \left(h(\mu) + \dfrac{\kappa\tau_{\kappa}}{4\left(3\lambda-1\right)}\right)c_{2}
	+ \left(h(\mu) - \dfrac{\kappa\tau_{\kappa}}{4\left(3\lambda-1\right)}\right)d_{2},
	\end{eqnarray}
	where
	\begin{eqnarray*}
		h(\mu) = \dfrac{\left(1-\mu\right)\kappa^{3}\tau_{\kappa}^{2}}{4\left[\lambda\left(2\lambda-1\right)\kappa^{2}\tau_{\kappa} + \left(\kappa - (\kappa^{2}+2)\tau_{\kappa}\right)\left(2\lambda-1\right)^{2}\right]}.
	\end{eqnarray*}
	Thus by taking modulus of \eqref{Th3-Fekete-e2}, we conclude that
	\begin{eqnarray*}
		\left|a_{3}-\mu a_{2}^{2}\right|
		\leq \left\{
		\begin{array}{ll}
			\dfrac{\left|\kappa\tau_{\kappa}\right|}{\left(3\lambda-1\right)} &; 
			0 \leq \left|h(\mu)\right|\leq \dfrac{\left|\kappa\tau_{\kappa}\right|}{4\left(3\lambda-1\right)}\\
			4\left|h(\mu)\right| &; 
			\left|h(\mu)\right|\geq \dfrac{\left|\kappa\tau_{\kappa}\right|}{4\left(3\lambda-1\right)}.
		\end{array}		
		\right.
	\end{eqnarray*}
	This gives desired inequality. 
\end{proof}
%%-----------------------------------------------------------------------
\begin{theorem}
	\label{thm5.1} Let $f(z)=z+\sum\limits_{n=2}^{\infty}a_{n}z^{n}$ be in the
	class $\mathcal{P}\mathcal{S}\mathcal{L}_{\Sigma}^{\kappa}(\lambda;\tilde{p}_{\kappa})$. Then 
	\begin{align*}
	\left|a_2\right|&\leq \dfrac{\left|\kappa\tau_{\kappa}\right|\sqrt{\kappa}}{\sqrt{\left[\kappa^{2}\tau_{\kappa}(1+2\lambda-\lambda^{2})+(\kappa-(\kappa^{2}+2)\tau_{\kappa})(1+\lambda)^{2}\right]}},
	\end{align*}
	\begin{align*}
	\left|a_3\right|&\leq 
	\dfrac{\left|\kappa\tau_{\kappa}\right|\left(\kappa-2(\kappa^{2}+1)\tau_{\kappa}\right)(1+\lambda)^{2}}{2(1+2\lambda)\left[\kappa^{2}\tau_{\kappa}(1+2\lambda-\lambda^{2})+(\kappa-(\kappa^{2}+2)\tau_{\kappa})(1+\lambda)^{2}\right]}
	\end{align*}
	and
	\begin{eqnarray*}
		\left|a_{3}-\mu a_{2}^{2}\right|
		\leq \left\{
		\begin{array}{ll}
			\dfrac{\left|\kappa\tau_{\kappa}\right|}{2+4\lambda};\\
			\qquad\qquad\qquad 
			0 \leq \left|\mu-1\right|\leq \dfrac{\left[\kappa^{2}\tau_{\kappa}(1+2\lambda-\lambda^{2})+(\kappa-(\kappa^{2}+2)\tau_{\kappa})(1+\lambda)^{2}\right]}{2\kappa^{2}\left|\tau_{\kappa}\right|(1+2\lambda)}\\
			\dfrac{\left|1-\mu\right|\kappa^{3}\tau_{\kappa}^{2}}{\left[\kappa^{2}\tau_{\kappa}(1+2\lambda-\lambda^{2})+(\kappa-(\kappa^{2}+2)\tau_{\kappa})(1+\lambda)^{2}\right]};\\
			\qquad\qquad\qquad\qquad 
			\left|\mu-1\right|\geq \dfrac{\left[\kappa^{2}\tau_{\kappa}(1+2\lambda-\lambda^{2})+(\kappa-(\kappa^{2}+2)\tau_{\kappa})(1+\lambda)^{2}\right]}{2\kappa^{2}\left|\tau_{\kappa}\right|(1+2\lambda)}.
		\end{array}		
		\right.
	\end{eqnarray*}
	%	where
	%	\begin{eqnarray*}
	%		M = %\left[\kappa^{2}\tau_{\kappa}(1+2\lambda-\lambda^{2})+(\kappa-(\kappa^{2}+2)\tau_{%\kappa})(1+\lambda)^{2}\right].
	%	\end{eqnarray*}
\end{theorem}

\begin{proof}
	Since $f\in\mathcal{P}\mathcal{S}\mathcal{L}_{\Sigma}^{\kappa}(\lambda;\tilde{p}_{\kappa})$, from the
	Definition \ref{def1.4} we have 
	\begin{equation}  \label{5.2}
	\frac{zf^{\prime}(z)+\lambda z^{2}f^{\prime\prime}(z)}{(1-\lambda)f(z)+\lambda zf^{\prime}(z)}
	= \widetilde{p_{\kappa}(u(z))}
	\end{equation}
	and 
	\begin{equation}  \label{5.3}
	\frac{wf^{\prime}(w)+\lambda w^{2}g^{\prime\prime}(w)}{(1-\lambda)g(w)+\lambda wg^{\prime}(w)}
	=\widetilde{p_{\kappa}(v(w))}.
	\end{equation}
	By virtue of \eqref{5.2}, \eqref{5.3}, \eqref{2.4} and \eqref{2.5}, we get 
	\begin{equation}  \label{5.6}
	(1+\lambda)a_2=\dfrac{c_{1}\kappa\tau_{\kappa}}{2},
	\end{equation}
	\begin{equation}  \label{5.7}
	2(1+2\lambda)a_{3}-(1+\lambda)^{2}a_{2}^{2}=\dfrac{1}{2}\left(c_{2}-\dfrac{c_{1}^{2}}{2}\right)\kappa\tau_{\kappa}+\dfrac{c_{1}^{2}}{4}(\kappa^{2}+2)\tau_{\kappa}^{2},
	\end{equation}
	
	\begin{equation}  \label{5.8}
	-(1+\lambda)a_2=\dfrac{d_{1}\kappa\tau_{\kappa}}{2}
	\end{equation}
	and 
	\begin{equation}  \label{5.9}
	-2(1+2\lambda)a_{3}-(\lambda^{2}-6\lambda-3)a_{2}^{2}  =\dfrac{1}{2}\left(d_{2}-\dfrac{d_{1}^{2}}{2}\right)\kappa\tau_{\kappa}+\dfrac{d_{1}^{2}}{4}(\kappa^{2}+2)\tau_{\kappa}^{2}.
	\end{equation}
	
	From \eqref{5.6} and \eqref{5.8}, we obtain
	\begin{equation*} 
	c_{1}=-d_{1},
	\end{equation*}
	and
	\begin{eqnarray} 
	2(1+\lambda)^{2} a_{2}^{2} &=& \dfrac{(c_{1}^{2} + d_{1}^{2})\kappa^{2}\tau_{\kappa}^{2}}{4}\nonumber\\
	a_{2}^{2} &=& \dfrac{(c_{1}^{2} + d_{1}^{2})\kappa^{2}\tau_{\kappa}^{2}}{8(1+\lambda)^{2}}~.\label{5.10}
	\end{eqnarray}
	By adding  \eqref{5.7} and \eqref{5.9}, we have 
	\begin{eqnarray}
	2\left(1+2\lambda-\lambda^{2}\right) a_{2}^{2}
	= \dfrac{1}{2}\left(c_{2}+d_{2}\right)\kappa\tau_{\kappa} - \dfrac{1}{4}\left(c_{1}^{2}+d_{1}^{2}\right)\kappa\tau_{\kappa} + \dfrac{1}{4}\left(c_{1}^{2}+d_{1}^{2}\right)(\kappa^{2}+2)\tau_{\kappa}^{2}.\label{5.11}
	\end{eqnarray}
	By substituting  \eqref{5.10} in \eqref{5.11}, we reduce that
	\begin{eqnarray}
	a_{2}^{2} &=& \dfrac{\left(c_{2}+d_{2}\right)\kappa^{3}\tau_{\kappa}^{2}}{4\left[\kappa^{2}\tau_{\kappa}(1+2\lambda-\lambda^{2})+(\kappa-(\kappa^{2}+2)\tau_{\kappa})(1+\lambda)^{2}\right]}.\label{5.12}
	\end{eqnarray}
	Now, applying Lemma \ref{lem-pom}, we obtain
	\begin{eqnarray}
	\left|a_{2}\right| &\leq& \dfrac{\left|\kappa\tau_{\kappa}\right|\sqrt{\kappa}}{\sqrt{\left[\kappa^{2}\tau_{\kappa}(1+2\lambda-\lambda^{2})+(\kappa-(\kappa^{2}+2)\tau_{\kappa})(1+\lambda)^{2}\right]}}.\label{5.13}
	\end{eqnarray}
	By subtracting \eqref{5.9} from \eqref{5.7}, we obtain
	\begin{eqnarray}
	a_{3}= \dfrac{\left(c_{2}-d_{2}\right)\kappa\tau_{\kappa}}{8\left(1+2\lambda\right)}  + a_{2}^{2}.\label{5.14}
	\end{eqnarray}
	Hence by Lemma \ref{lem-pom}, we have 
	\begin{eqnarray*}
		\left|a_{3}\right|&\leq& \dfrac{\left(\left|c_{2}\right|+\left|d_{2}\right|\right)\left|\kappa\tau_{\kappa}\right|}{8\left(1+2\lambda\right)}  + \left|a_{2}\right|^{2}\leq  \dfrac{\left|\kappa\tau_{\kappa}\right|}{2+4\lambda} + \left|a_{2}\right|^{2}
		.\label{5.15}
	\end{eqnarray*}
	Then in view of \eqref{5.13}, we obtain
	\begin{eqnarray*}
		\left|a_{3}\right|
		\leq 
		\dfrac{\left|\kappa\tau_{\kappa}\right|\left(\kappa-2(\kappa^{2}+1)\tau_{\kappa}\right)(1+\lambda)^{2}}{2(1+2\lambda)\left[\kappa^{2}\tau_{\kappa}(1+2\lambda-\lambda^{2})+(\kappa-(\kappa^{2}+2)\tau_{\kappa})(1+\lambda)^{2}\right]}
	\end{eqnarray*}
	From \eqref{5.14}, we have 
	\begin{eqnarray}\label{Th4-Fekete-e1}
	a_{3} -\mu a_{2}^{2} = \dfrac{\left(c_{2}-d_{2}\right)\kappa\tau_{\kappa}}{8\left(1+2\lambda\right)} + \left(1-\mu\right)a_{2}^{2}.
	\end{eqnarray}
	By substituting \eqref{5.12} in \eqref{Th4-Fekete-e1}, we have 
	\begin{eqnarray}\label{Th4-Fekete-e2}
	a_{3} -\mu a_{2}^{2} 
	&=& \dfrac{\left(c_{2}-d_{2}\right)\kappa\tau_{\kappa}}{8\left(1+2\lambda\right)} + \left(1-\mu\right) \left(\dfrac{\left(c_{2}+d_{2}\right)\kappa^{3}\tau_{\kappa}^{2}}{4\left[\kappa^{2}\tau_{\kappa}(1+2\lambda-\lambda^{2})+(\kappa-(\kappa^{2}+2)\tau_{\kappa})(1+\lambda)^{2}\right]}\right)\nonumber\\
	&=& \left(h(\mu) + \dfrac{\kappa\tau_{\kappa}}{8\left(1+2\lambda\right)}\right)c_{2}
	+ \left(h(\mu) - \dfrac{\kappa\tau_{\kappa}}{8\left(1+2\lambda\right)}\right)d_{2},
	\end{eqnarray}
	where
	\begin{eqnarray*}
		h(\mu) = \dfrac{\left(1-\mu\right)\kappa^{3}\tau_{\kappa}^{2}}{4\left[\kappa^{2}\tau_{\kappa}(1+2\lambda-\lambda^{2})+(\kappa-(\kappa^{2}+2)\tau_{\kappa})(1+\lambda)^{2}\right]}.
	\end{eqnarray*}
	Thus by taking modulus of \eqref{Th4-Fekete-e2}, we conclude that
	\begin{eqnarray*}
		\left|a_{3}-\mu a_{2}^{2}\right|
		\leq \left\{
		\begin{array}{ll}
			\dfrac{\left|\kappa\tau_{\kappa}\right|}{2\left(1+2\lambda\right)} &; 
			0 \leq \left|h(\mu)\right|\leq \dfrac{\left|\kappa\tau_{\kappa}\right|}{8\left(1+2\lambda\right)}\\
			4\left|h(\mu)\right| &; 
			\left|h(\mu)\right|\geq \dfrac{\left|\kappa\tau_{\kappa}\right|}{8\left(1+2\lambda\right)}.
		\end{array}		
		\right.
	\end{eqnarray*}
	This gives desired inequality. 
\end{proof}
%%================================================================================

\section{Corollaries and Consequences}
\begin{corollary}
	\label{cor2.1} Let $f(z)=z+\sum\limits_{n=2}^{\infty}a_{n}z^{n}$ be in the
	class $\mathcal{F}\mathcal{S}\mathcal{L}_{\Sigma}^{\kappa}(\gamma, \lambda, \tilde{p}_{\kappa})$. Then 
	\begin{align*}
	\left|a_{2}\right| &\leq \dfrac{\left|\gamma\right|\left|\kappa\tau_{\kappa}\right|\sqrt{\kappa}}{\sqrt{3\gamma \kappa^{2}\tau_{\kappa}\left(1+ 2\lambda\right)+4(\kappa-(\kappa^{2}+2)\tau_{\kappa})(1+\lambda)^{2}}},\qquad
	\end{align*}
	\begin{align*}
	\left|a_3\right|\leq \dfrac{4\left|\gamma\right|\left|\kappa\tau_{\kappa}\right|(\kappa-(\kappa^{2}+2)\tau_{\kappa})(1+\lambda)^{2}}{3(1+2\lambda)\left[3\gamma \kappa^{2}\tau_{\kappa}\left(1+ 2\lambda\right)+4(\kappa-(\kappa^{2}+2)\tau_{\kappa})(1+\lambda)^{2}\right]}
	\end{align*}
	and
	\begin{eqnarray*}
	\left|a_{3}-\mu a_{2}^{2}\right|
	\leq \left\{
	\begin{array}{ll}
	\dfrac{\left|\gamma\right|\left|\kappa\tau_{\kappa}\right|}{3+6\lambda} ;
	\\
	\qquad\qquad 
	0 \leq \left|\mu - 1\right|\leq \dfrac{3\gamma \kappa^{2}\tau_{\kappa}\left(1+ 2\lambda\right)+4(\kappa-(\kappa^{2}+2)\tau_{\kappa})(1+\lambda)^{2}}{\left(3+6\lambda\right)\left|\gamma\right|\kappa^{2}\left|\tau_{\kappa}\right|}
	\\
	\dfrac{\left|1-\mu\right|\gamma^{2}\kappa^{3}\tau_{\kappa}^{2}}{3\gamma \kappa^{2}\tau_{\kappa}\left(1+ 2\lambda\right)+4(\kappa-(\kappa^{2}+2)\tau_{\kappa})(1+\lambda)^{2}} ; 
	\\
	\qquad\qquad
	\left|\mu -1\right|\geq \dfrac{3\gamma \kappa^{2}\tau_{\kappa}\left(1+ 2\lambda\right)+4(\kappa-(\kappa^{2}+2)\tau_{\kappa})(1+\lambda)^{2}}{\left(3+6\lambda\right)\left|\gamma\right|\kappa^{2}\left|\tau_{\kappa}\right|}~.
	\end{array}		
	\right.
	\end{eqnarray*}
%	where
%	\begin{eqnarray*}
%		h(\mu) = \dfrac{\left(1-\mu\right)\gamma^{2}\tau^{2}}{4\left[3\gamma\tau(1+2\lambda)+4(1-3\tau)(1+\lambda)^{2}\right]}.
%	\end{eqnarray*}
\end{corollary}

\begin{corollary}
	\label{cor2.2} Let $f(z)=z+\sum\limits_{n=2}^{\infty}a_{n}z^{n}$ be in the
	class $\mathcal{B}\mathcal{S}\mathcal{L}_{\Sigma}^{\kappa}(\gamma, \alpha, \tilde{p}_{\kappa})$. Then 
	\begin{align*}
	\left|a_{2}\right| &\leq \dfrac{\left|\gamma\right|\left|\kappa\tau_{\kappa}\right|\sqrt{\kappa}}{\sqrt{\gamma \kappa^{2}\tau_{\kappa}\left(1+2\alpha\right)+(\kappa-(\kappa^{2}+2)\tau_{\kappa})(1+\alpha)^{2}}},\qquad
	\end{align*}
	\begin{align*}
	\left|a_3\right|\leq \dfrac{\left|\gamma\right|\left|\kappa\tau_{\kappa}\right|(\kappa-(\kappa^{2}+2)\tau_{\kappa})(1+\alpha)^{2}}{(1+2\alpha)\left[\gamma \kappa^{2}\tau_{\kappa}\left(1+2\alpha\right)+(\kappa-(\kappa^{2}+2)\tau_{\kappa})(1+\alpha)^{2}\right]}
	\end{align*}
	and
	\begin{eqnarray*}
		\left|a_{3}-\mu a_{2}^{2}\right|
		\leq \left\{
		\begin{array}{ll}
			\dfrac{\left|\gamma\right|\left|\kappa\tau_{\kappa}\right|}{1+2\alpha};\\
			\qquad\qquad\qquad 
			0 \leq \left|\mu-1\right|\leq \dfrac{\gamma \kappa^{2}\tau_{\kappa}\left(1+2\alpha\right)+(\kappa-(\kappa^{2}+2)\tau_{\kappa})(1+\alpha)^{2}}{\left(1+2\alpha\right)\left|\gamma\right|\kappa^{2}\left|\tau_{\kappa}\right|}\\
			\dfrac{\left|1-\mu\right|\gamma^{2}\kappa^{3}\tau_{\kappa}^{2}}{\gamma \kappa^{2}\tau_{\kappa}\left(1+2\alpha\right)+(\kappa-(\kappa^{2}+2)\tau_{\kappa})(1+\alpha)^{2}};\\
			\qquad\qquad\qquad 
			\left|\mu-1\right|\geq \dfrac{\gamma \kappa^{2}\tau_{\kappa}\left(1+2\alpha\right)+(\kappa-(\kappa^{2}+2)\tau_{\kappa})(1+\alpha)^{2}}{\left(1+2\alpha\right)\left|\gamma\right|\kappa^{2}\left|\tau_{\kappa}\right|}~.
		\end{array}		
		\right.
	\end{eqnarray*}
\end{corollary}

\begin{corollary}
	\label{cor2.3} Let $f(z)=z+\sum\limits_{n=2}^{\infty}a_{n}z^{n}$ be in the
	class $\mathcal{H}\mathcal{S}\mathcal{L}_{\Sigma}^{\kappa}(\gamma, \tilde{p}_{\kappa})$. Then 
	\begin{align*}
	\left|a_{2}\right| &\leq \dfrac{\left|\gamma\right|\left|\kappa\tau_{\kappa}\right|\sqrt{\kappa}}{\sqrt{3\gamma \kappa^{2}\tau_{\kappa}+4(\kappa-(\kappa^{2}+2)\tau_{\kappa})}},
	\qquad
	\left|a_3\right|\leq \dfrac{4\left|\gamma\right|\left|\kappa\tau_{\kappa}\right|(\kappa-(\kappa^{2}+2)\tau_{\kappa})}{3\left[3\gamma \kappa^{2}\tau_{\kappa}+4(\kappa-(\kappa^{2}+2)\tau_{\kappa})\right]},
	\end{align*}
	and
	\begin{eqnarray*}
		\left|a_{3}-\mu a_{2}^{2}\right|
		\leq \left\{
		\begin{array}{ll}
			\dfrac{\left|\gamma\right|\left|\kappa\tau\right|}{3} &; 
			0 \leq \left|\mu - 1\right|\leq \dfrac{3\gamma \kappa^{2}\tau_{\kappa}+4(\kappa-(\kappa^{2}+2)\tau_{\kappa})}{3\left|\gamma\right|\kappa^{2}\left|\tau_{\kappa}\right|}\\
			\dfrac{\left|1-\mu\right|\gamma^{2}\kappa^{3}\tau_{\kappa}^{2}}{3\gamma \kappa^{2}\tau_{\kappa}+4(\kappa-(\kappa^{2}+2)\tau_{\kappa})} &; 
			\left|\mu-1\right|\geq \dfrac{3\gamma \kappa^{2}\tau_{\kappa}+4(\kappa-(\kappa^{2}+2)\tau_{\kappa})}{3\left|\gamma\right|\kappa^{2}\left|\tau_{\kappa}\right|}~.
		\end{array}		
		\right.
	\end{eqnarray*}
\end{corollary}
%%--------------------------------------------------------------------------------
\begin{corollary}
	\label{Cor-Thm3.1} Let $f(z)=z+\sum\limits_{n=2}^{\infty}a_{n}z^{n}$ be in the
	class $\mathcal{S}\mathcal{L}_{\Sigma}^{\kappa}(\gamma,\tilde{p}_{\kappa})$. Then 
	\begin{align*}
	\left|a_2\right|&\leq 
	\dfrac{\left|\gamma\right|\left|\kappa\tau_{\kappa}\right|\sqrt{\kappa}}{\sqrt{\gamma \kappa^{2}\tau_{\kappa}+(\kappa - (\kappa^{2}+2)\tau_{\kappa})}},\quad
	%\end{align*}
	%\begin{align*}
	\left|a_3\right|\leq \dfrac{\left|\gamma\right|\left|\kappa\tau_{\kappa}\right|\left|(\kappa - (\kappa^{2}+2)\tau_{\kappa})-\gamma\kappa^{2}\tau_{\kappa}\right|}{2\gamma\kappa^{2}\tau_{\kappa}+2(\kappa - (\kappa^{2}+2)\tau_{\kappa})}
	\end{align*}
	and
\begin{eqnarray*}
	\left|a_{3}-\mu a_{2}^{2}\right|
	\leq \left\{
	\begin{array}{ll}
		\dfrac{\left|\gamma\right|\left|\kappa\tau_{\kappa}\right|}{2} &; 
		0 \leq \left|\mu-1\right|\leq \dfrac{\gamma\kappa^{2}\tau_{\kappa}+\left(\kappa-(\kappa^{2}+2)\tau_{\kappa}\right)}{2\left|\gamma\right|\kappa^{2}\left|\tau_{\kappa}\right|}\\
		\dfrac{\left|1-\mu\right|\gamma^{2}\kappa^{3}\tau_{\kappa}^{2}}{\gamma\kappa^{2}\tau_{\kappa}+\left(\kappa-(\kappa^{2}+2)\tau_{\kappa}\right)} &; 
		\left|\mu-1\right|\geq \dfrac{\gamma\kappa^{2}\tau_{\kappa}+\left(\kappa-(\kappa^{2}+2)\tau_{\kappa}\right)}{2\left|\gamma\right|\kappa^{2}\left|\tau_{\kappa}\right|}.
	\end{array}		
	\right.
\end{eqnarray*}
\end{corollary}
%%--------------------------------------------------------------------------------
%Taking $\gamma=1$ and $\lambda=0$ in Theorem \ref{thm3.1},  $\lambda=1$ in Theorem \ref{thm4.1} and $\lambda=0$ in Theorem \ref{thm5.1} we have the following corollary. 
\begin{corollary}\cite{HOG-GMS-JS-2019-CFS-kF}
	\label{cor4.1} Let $f(z)=z+\sum\limits_{n=2}^{\infty}a_{n}z^{n}$ be in the
	class $\mathcal{S}\mathcal{L}_{\Sigma}^{\kappa}(\tilde{p}_{\kappa}).$ Then 
	\begin{align*}
	\left|a_2\right|&\leq \dfrac{\left|\kappa\tau_{\kappa}\right|\sqrt{\kappa}}{\sqrt{\kappa-2\tau_{\kappa}}}, \;\;\;\; 
	\left|a_3\right| \leq \dfrac{\left|\kappa\tau_{\kappa}\right|(\kappa-2(\kappa^{2}+1)\tau_{\kappa})}{2\kappa-4\tau_{\kappa}}
	\end{align*}
	and
	\begin{eqnarray*}
	\left|a_{3}-\mu a_{2}^{2}\right|
	\leq \left\{
	\begin{array}{ll}
	\dfrac{\left|\kappa\tau_{\kappa}\right|}{2} &; 
	0 \leq \left|\mu-1\right|\leq \dfrac{\kappa-2\tau_{\kappa}}{2\kappa^{2}\left|\tau_{\kappa}\right|}\\
	\dfrac{\left|1-\mu\right|\kappa^{3}\tau_{\kappa}^{2}}{\kappa-2\tau_{\kappa}} &; 
	\left|\mu-1\right|\geq \dfrac{\kappa-2\tau_{\kappa}}{2\kappa^{2}\left|\tau_{\kappa}\right|}.
	\end{array}		
	\right.
	\end{eqnarray*}
\end{corollary}
%%-----------------------------------------------------------------------
%Taking $\lambda=1$ in Theorem \ref{thm5.1} we have the following corollary. 
\begin{corollary}\cite{HOG-GMS-JS-2019-CFS-kF}
	\label{cor4.2} Let $f(z)=z+\sum\limits_{n=2}^{\infty}a_{n}z^{n}$ be in the
	class $\mathcal{K}\mathcal{S}\mathcal{L}_{\Sigma}^{\kappa}(\tilde{p}_{\kappa}).$ Then 
	\begin{align*}
	\left|a_2\right|&\leq \dfrac{\left|\kappa\tau_{\kappa}\right|\sqrt{\kappa}}{\sqrt{2(2\kappa-(\kappa^{2}+4)\tau_{\kappa})}}, \;\;\;\; 
	\left|a_3\right| \leq \dfrac{\left|\kappa\tau_{\kappa}\right|(\kappa-2(\kappa^{2}+1)\tau_{\kappa})}{3(2\kappa-(\kappa^{2}+4)\tau_{\kappa})}.
	\end{align*}
	and
	\begin{eqnarray*}
	\left|a_{3}-\mu a_{2}^{2}\right|
	\leq \left\{
	\begin{array}{ll}
	\dfrac{\left|\kappa\tau_{\kappa}\right|}{6} &; 
	0 \leq \left|\mu-1\right|\leq \dfrac{2\kappa-(\kappa^{2}+4)\tau_{\kappa}}{3\kappa^{2}\left|\tau_{\kappa}\right|}\\
	\dfrac{\left|1-\mu\right|\kappa^{3}\tau_{\kappa}^{2}}{2(2\kappa-(\kappa^{2}+4)\tau_{\kappa})} &; 
	\left|\mu-1\right|\geq \dfrac{2\kappa-(\kappa^{2}+4)\tau_{\kappa}}{3\kappa^{2}\left|\tau_{\kappa}\right|}.
	\end{array}		
	\right.
	\end{eqnarray*}
\end{corollary}
\begin{remark}
	Results discussed in Corollaries \ref{cor4.1} and \ref{cor4.2} are coincide with bounds obtained in \cite{HOG-GMS-JS-2019-CFS-kF}. Also, For $\kappa=1$, all the results obtained are coincides with results obtained in \cite{NM-VKB-CA-2016}.
\end{remark}
%%--------------------------
\section*{Declarations}
\
\\
\
\textbf{Availability of data and material~:~} No data were used to support this study.
\
\\\\
\
\textbf{Competing interests~:~} The authors declare that there are no conflicts of interest regarding the publication of this manuscript. 
\
\\\\
\
\textbf{Funding ~:~} Not applicable.
\
\\\\
\
\textbf{Authors' contributions ~:~}  All authors are equally contribute in writing this manuscript. All authors read and approved the final manuscript."
\
\\\\
\
\textbf{Acknowledgements ~:~} Not applicable.

%-------References------------

\end{document}